\documentclass[a4paper]{amsart}
\usepackage{amssymb}

\usepackage{color}
\newtheorem{theo}{Theorem}[section]
\newtheorem{prop}{Proposition}[section]

\newtheorem{lemma}{Lemma}[section]
\newtheorem{rema}{Remark}[section]

\numberwithin{equation}{section} 

\def\t{\textrm} 

\def\beq{\begin{eqnarray}}
\def\eeq{\end{eqnarray}}
\def\baa{\begin{array}}
\def\eaa{\end{array}}

\newcommand{\bdef}{\begin{definition}}
\newcommand{\be}{\begin{equation}}
\newcommand{\ee}{\end{equation}}
\newcommand{\bt}{\begin{theo}}
\newcommand{\et}{\end{theo}}
\newcommand{\bl}{\begin{lemma}}
\newcommand{\el}{\end{lemma}}
\newcommand{\bp}{\begin{prop}}
\newcommand{\epr}{\end{prop}}

\def\C{C^f}
\def\dsp{\displaystyle}
\def\pa{\partial}
\def\ep{\epsilon}
\def\ov{\overline}

\let\dsp=\displaystyle
\def\R{{\mathbb R}}

\def\li{\lambda_i}
\def\lai{\lambda_i}
\def\v{v_i}
\def\u{u_i}
\def\vx{{v_i}_x}
\def\ux{{u_i}_x}

\def\vt{{v_i}_t}
\def\ut{{u_i}_t}

\def\beti{\beta_i}
\def\Di{D_i}
\def\phii{\phi_i}
\def\psii{\psi_i}
\def\Psii{\Psi_i}
\def\phx{{\phi_i}_x}
\def\psix{{\psi_i}_x}
\def\Psix{{\Psi_i}_x}

\def\phxx{{\phi_i}_{xx}}

\def\pho{\phi_{0i}}

\def\uo{u_{0i}}
\def\vo{v_{0i}}

\def\vox{{v_{0i}}_x}
\def\hv{\hat v}
\def\M{{\mathcal{M}}}
\def\U{\mathcal{O}}
\def\E{\mathcal{I}}
\def\A{\mathcal{A}}
\def\U{\mathcal{U}}
\def\Ii{I_i}

\def\iIi{\int_{\Ii}}
\def\nld{\Vert_{2}}

\def\nhu{\Vert_{H^1}}

\def\soT{\sup_{[0,T]}}
\def\iot{\int_0^t}

\def\nhd{\Vert_{H^2}}
\def\ioT{\int_0^T}

\def\suM{\sum_{i\in \M}}

\def\suEn{\sum_{i\in \E^\nu}}
\def\suUn{\sum_{i\in \mathcal O^\nu}}
\def\suMj{\sum_{j\in \M}}
\def\suMjj{\sum_{j\in \mathcal J}}
\def\suMjn{\sum_{j\in \M^\nu}}

\def\del{\Delta^h}
\def\T{\mathcal{T}}

\begin{document}

\title[]
{Global solutions  for a chemotaxis hyperbolic-parabolic  system on networks with nonhomogeneous boundary conditions}
\author[ Francesca R. Guarguaglini]{ Francesca R. 
Guarguaglini$^\diamond$}
 
\thanks{ \noindent 
$\diamond$ Dipartimento di Ingegneria e Scienze dell'Informazione e Matematica, Universit\'a degli
Studi di L'Aquila, Via Vetoio, I--67100 Coppito (L'Aquila), Italy. E--mail: france.gua@gmail.com
}
\date{}

\begin{abstract} 
In this  paper we study a semilinear hyperbolic-parabolic system as a model for  some chemotaxis  phenomena evolving on  networks; we consider transmission conditions at the inner nodes which preserve the fluxes and nonhomogeneous boundary conditions
having  in mind phenomena with inflow of cells and food providing at the network exits. 
We give some conditions on the boundary data which ensure the existence of stationary solutions and we prove that these ones are asymptotic profiles for a class of global solutions.
\end{abstract}
\maketitle
\bigskip

\begin{section}
{\bf Introduction}

In this paper we consider the  one dimensional semilinear hyperbolic-parabolic system 
\begin{equation}\label{is}
    \left\{ \begin{array}{l}
\partial_tu +\lambda\partial_x v=0\ ,
\\ \\
\partial_t v +\lambda \partial_x u= u \partial_x \psi -\beta v\ ,
\\ \\
\partial_t \psi= D\partial_{xx} \psi +a u -b\psi\ ,
\end{array}\right.\end{equation}
on a  finite planar network, where $\lambda, \beta, D,b>0$ and $a\geq 0$ .

 The system    has been proposed as a model for  chemosensitive movements of bacteria or cells; the unknown $u$ stands for the cells concentration, $\lambda v$ denotes  their average flux and $\psi$ is the chemo-attractant concentration produced by the cells themselves;
 the individuals  move at a constant velocity, whose modulus is $\lambda$, towards the right or left along the axis; $\beta$ is the friction coefficient while $D,a,b$ are respectively  the diffusion coefficient, the production rate and the degradation one for the chemoattractant .
 
 Systems like (\ref{is}) are   adaptations  of the so-called Cattaneo equation to the chemotactic case, introducing the nonlinear term $u\psi_x$ in the equation for the flux \cite{Segel, AG}, and their solutions have been  studied in \cite{hillerho,hillenst,gumanari}; they are  included among  hyperbolic models which have been recently introduced in contrast to the parabolic ones considered before,
 since they give rise to  a finite speed of propagation and allow better observation of the phenomena during the initial phase.
 
 In recent years, one dimensional  models  on networks  have been developed in order to describe particular chemotactic phenomena like
the process of dermal wound healing and the behavior of the slime mold  Physarum polycephalum as a model for amoeboid movements.
Actually, during the healing process, the stem cells in charge of the reparation of dermal tissue (fibroblasts), create a new extracellular matrix, essentially made by collagen, and move along it to fill the wound driven by chemotaxis
and  tissue engineers insert artificial scaffolds within the wound to accelerate this process
 \cite{harley, mandal, spadac}; also,  the body of Physarum polycephalum
contains a network of tubes  which are used by nutrients and chemical signals to circulate throughout the organism \cite{naka}.

These models are heavily characterize by the transmission conditions set at the internal nodes of the network, which couple the solutions on different arcs.

Here  we consider the system (\ref{is}) on  a network  whose arcs $I_i$ are  characterized by the parameters $\li,\beta_i,D_i,a_i, b_i$.   The triples of unkowns $(\u,\v,\psi_i)$ corresponding to each arc are coupled by the transmission conditions introduced in  \cite{noi} set at the inner nodes, which impose that the sum of the incoming fluxes equals the sum of the outgoing ones, rather than the continuity of the densities, since the eventuality of jumps at the nodes for these quantities   seems a more appropriate framework to describe movements of individuals.

This model, complemented with homogeneous boundary conditions at the external vertices of the network, was studied in
\cite{noi}, concerning the existence and the uniqueness of global solutions 
 in the case of suitably small initial data; moreover, results about existence of stationary solutions and asymptotic behaviour are given in \cite{G}; finally, 
  in \cite{BNR}  the authors carry out a numerical study of the same system with transmission conditions set for
the Riemann invariants of the hyperbolic part, which are equivalent to our ones for some choices of the coefficients.

Results about hyperbolic models on networks can be found in \cite{ zua1,pic, zua2,nik,zong}, with different kinds of transmission conditions; moreover parabolic chemotaxis models on networks were studied in \cite{Bor,CC,mugn}, with continuity conditions at the nodes.

It is worth considering  system (\ref{is}) with nonhomogeneous  boundary conditions at the  outer nodes of  the network, having  in mind phenomena with inflow of cells and food providing at the network exits,  in particular experiments 
on the behaviour of  Physarum   \cite{naka}. 
We remark that \cite{BN} contains a numerical approach to system (\ref{is}) on networks, with transmission conditions given for the Riemann invariants and nonhomogeneous  conditions at the boundaries;   the numerical tests show the 
correspondence with  the main features of the 
real behavior of slime mold examined  through the laboratory experiments:
dead end cutting and the selection of the solution path among the competitive paths.

So, in the present paper  we consider system (\ref{is}) with the dissipative transmission conditions introduced in \cite{noi} at the inner nodes,  and nonhomogeneous Neumann conditions for the hyperbolic part and nonhomogeneous Robin condition for the parabolic equation at the external ones.
The boundary data are assumed to satisfy suitable hypothesis   ensuring, in particular,  the boundedness of the total mass of cells during the phenomenon evolution; the mass is preserved in case of homogeneous Neumann  conditions, since   the conservation of the fluxes holds at each  inner nodes, due to the transmission conditions \cite{noi,G}, but in the present case it depends on the evolution in time of the boundary values for the fluxes $\li v_i$.

 The first result in the  paper is the existence of local solutions; it is achieved by linear contraction semigroups  together with the abstract theory for semilinear problems, and the dissipative transmission conditions at the inner nodes play a fundamental role.
 
 The existence of global solutions is achieved under  assumptions of smallness of the data, proceeding in some steps. 
 First we  assume the existence of a stationary solution $(U(x),V(x),\Psi(x))$ to the problem and we obtain 
  a priori estimates for solutions corresponding to initial and boundary data which are small  perturbations of  the possible stationary solution. Here a fondamental role is played by a  suitable condition stated for  the transmission coefficients, which 
  allows to express the jumps of the densitiy $u$ at each inner node as linear combinations of the values of the fluxes  at the same node. This fact and assumptions on the data provide a  control of 
   the evolution in time of the $L^\infty$- norm of the density  which permits
  to remove some conditions on the parameters $a_i$ and $b_i$ considered in \cite{noi,G}.
 When the boundary data for the fluxes $\li \v$ are constant functions, the hypothesys necessary to prove the a priori estimates imply that the sum of the fluxes incoming in the network have to equal the sum of the outgoing ones and that the initial mass of cells has to equal the mass of the stationary solution.
 
 If a stationary solution $(U(x),V(x),\Psi(x))$ exists and the quantities $\Vert U\Vert_\infty$ and $\Vert \Psi_x\Vert_\infty$ are small, the a priori estimates provide a  bound, uniform in time, for a norm of the solutions having small perturbations of  the  stationary one as initial and boundary data; in this  way, after the proof of real existence of stationary solutions,  we would obtain
  the  existence of global solutions for a class of initial and boundary data and would identify   the stationary solutions as the asymptotic profiles for such class of solutions.

 For this reason we devote part of this paper to study the existence of stationary solutions. In the cases of acyclic networks we prove two results, under different smallness conditions on the boundary  data and on the total mass; in particular we give conditions which ensure the existence of a stationary solution with non-negative density $U$.
 For general networks we exhibit some stationary solutions in very particular cases for the parameters of the problem.
 
 The paper is organized as follows. In Section 2 we give the statement of the problem and, in particular, we introduce  the transmission conditions and the assumption on the  data,  while in Section 3 we prove the local existence result.
Section 4 is devoted to  the  a priori estimates and to the consequent global existence and asymptotic behaviour results, under the assumption that a {\it small } stationary solution exists. In Section 5 we prove the  results of  existence of stationary solutions  in the case of acyclic networks. Finally, in Section 6 we present the global existence and asymptotic behaviour results under assumptions which ensure  the  real existence of stationary solutions.

 \end{section}

\begin{section}
{\bf Statement of the problem }

We consider a planar finite connected graph $\mathcal G=(\mathcal Z, \mathcal A)$ composed by a 
set  $\mathcal Z$ 
of $n$ nodes (or vertexes)
 and a set $\mathcal A$
of $m$ oriented arcs, $\mathcal A=\{I_i:i\in \mathcal M=\{1,2,...,m\}\}$. 
Each node 
 is a point of the plane and 
each oriented arc $I_i$ is an oriented segment  joining two nodes.

We use   $e_j$,  $j\in \mathcal J$,  to indicate the external vertexes of the graph, i.e. the vertexes belonging  to only one  arc, and by
$I_{i(j)}$
  the  external arc  incident with $e_j$.
  Moreover, we denote by   $N_\nu$, $\nu\in\mathcal N$, the internal nodes; for each of them we consider the set of incoming arcs 
$\mathcal A_{in}^\nu=\{I_i:i\in \mathcal I^\nu\}$
and the set of the outgoing ones
$\mathcal A_{out}^\nu=\{I_i:i\in \mathcal O^\nu\}$; let $\mathcal M^\nu=\mathcal I^\nu\cup\mathcal O^\nu$.

In this paper, a  {\it path} in the  graph is  a sequence of arcs, two by two adjacent,
without taking into account  orientations. Moreover, we call {\it acyclic} a graph which does not contains cycles, i.e. for each couple of nodes there exists a unique  path connecting them, whose arcs are covered only one time.

Each arc $I_i$ is considered as a one dimensional interval $(0,L_i)$.
A function $f$ defined on $\mathcal A$ is a m-tuple of functions $f_i$, $i\in \mathcal M$, each one defined 
on $I_i$. The expression $ f_i(N_\nu)$ means $f_i(0) $ if $N_\nu$ is the starting point of the arc $I_i$ and $f_i(L_i)$ if $N_\nu$ is the endpoint, and similarly for $f(e_j)$.

We set $\dsp L^p(\mathcal A):=\{f: f_i\in L^p(I_i)\}$, 
$\dsp H^s(\mathcal A):=\{f :f_i\in H^s(I_i)\}$
 and
$$\Vert f\Vert_2:=\suM \Vert f_i\Vert_2\ , \ \Vert f\Vert_\infty:=\sup_{i\in\M} \Vert f_i\Vert_\infty\ ,\ 
\Vert f\Vert_{H^s}:=\suM \Vert f_i\Vert_{H^s}.$$ 

We consider the evolution  of the following problem on the graph $\mathcal G$
\begin{equation}\label{sysi}
    \left\{ \begin{array}{l}
\partial_t \u +\lambda_i \partial_x \v=0\ ,
\\ \\
\partial_t \v +\lambda_i \partial_x \u= \u \partial_x \psi_i -\beta_i \v\ ,
\qquad     t\geq 0  , \ x\in I_i  , \  i\in \M,
\\ \\
\partial_t \psii= D_i \partial_{xx} \psi_i +a_i\u -b_i\psi_i\ ,
\end{array}\right.\end{equation}
where $ a_i\geq 0\ ,\li\ b_i,D_i,\beta_i >0 $ , 
complemented with the initial conditions
\be\label{uvic}
(u_{i0},v_{i0})\in (H^1(I_i))^2\ ,\ 
\psi_{i0}\in H^2(I_i)
\ \t{ for } i\in\M\  .\ee

In order to set boundary and transmission conditions, we introduce the following parameters: 
$$\eta_j=\left \{ \begin{array}{ll} 1& \t{ if the arc } I_{i(j)}  \t{ is incoming in } e_j\ ,  \ \ \quad    j\in \mathcal J, \\ 
-1&  \t{ if the arc } I_{i(j)} \t{ is outgoing from } e_j\ , \ \ j\in \mathcal J,\end{array} \right. $$
$$\delta_i^\nu = 1\ \t{ if  }i\in \mathcal I^\nu \ ,\  \delta _i^\nu =-1 \t{ if } i\in\mathcal O^\nu\  ,\ \ \nu\in\mathcal N. \qquad \qquad \qquad$$

The boundary conditions  for $v$, at each outer point $e_j$, are 
\be\label{vbc}
\eta_j\lambda_{i(j)} v_{i(j)}(e_j, t)=\mathcal W_j(t)\in   W^{2,1}(0,\ov T)\ ,
\qquad 
\  \ov T>0\ , \ j\in\mathcal J\  ,
\ee
while for $\psi$ we set  the Robin boundary conditions 
\be\label{phibc}
\eta_j D_{i(j)} {\partial_x\psi_{i(j)}}(e_j, t) + d_j \psi_{i(j)}(e_j,t)
=\mathcal  P_j(t)\in        H^{2}(0,\ov T)  
 \ ,\   d_j\geq 0\ ,\ \ov T>0\ , \ j\in \mathcal J .\ee



In addition, at each internal node $N_\nu$ we impose the following transmission conditions for the unknown $\psi$
\be\label{KC}\left\{\begin{array}{ll}\displaystyle
\delta_i^\nu \Di \partial_x \psi_i(N_\nu,t)=  \sum_{j\in\M^\nu} \alpha_{ij}^\nu(\psi_j(N_\nu,t)-\psii(N_\nu,t))\  \ ,\ i\in \M^\nu \ ,\  t> 0\ ,
\\ 
\alpha_{ij}^\nu\geq 0\ ,\ \alpha_{ij}^\nu=\alpha_{ji}^\nu \ \t{ for all } i,j\in \M^\nu\ ,
\end{array}\right.\ee
and the following ones for  the unknowns $v$ and $u$ 
\be\label{TC}\left\{\begin{array}{ll}\dsp
- \delta_i^\nu \lai\v(N_\nu,t)=\sum_{j\in\M^\nu}
\sigma_{ij}^\nu \left(u_j(N_\nu,t)-\u(N_\nu,t)\right)\ ,\ i\in \M^\nu\ ,\ t>0\ ,\\ 
\sigma_{ij}^\nu\geq 0\ ,\sigma_{ij}^\nu=\sigma_{ji}^\nu \ \t{ for all } i,j\in \M^\nu\ .
\end{array}\right.\ 
\ee

Motivations for the above constraints on the coefficients in the transmission conditions can be found in \cite{noi} . These kind of transmission conditions were introduced in \cite{KK} in a parabolic model for the description of passive transport through biological membranes and they are known as Kedem-Katchalsky permeability conditions. 

Finally, we impose the following compatibility conditions
 \be\label{compat} 
u_{i0},v_{i0}, \psi_{i0}\t{  satisfy conditions (\ref{vbc})-(\ref{TC})  
for all } i\in\M\ .
\ee

First we are going to prove that the problem (\ref{sysi})-(\ref{compat}) has a unique local solution
 $$u,v \in C([0,T];H^1(\mathcal A))\cap C^1([0,T];L^2(\mathcal A))\ ,$$
$$\psi \in C([0,T];H^2(\mathcal A))\cap C^1([0,T];L^2(\mathcal A))\cap H^1((0,T);H^1(\A))\ ,$$ 
for some $T>0$.

On the other hand, the proofs of the existence  of global solutions
and of the existence  of stationary solutions on acyclic graphs, carried out in the last sections,  require the following further 
 conditions on the transmission coefficients
\be\label{nd} \t{ for all }\nu\in\mathcal N, \t{ for some } k\in\M^\nu, 
\sigma_{ik}^\nu\neq 0 \t{  for all } i\in\M^\nu, i\neq k\ ,
\ee
in addition to suitable smallness and smoothness assumptions on the data.

Finally, we remark that the transmission conditions (\ref{KC})
 imply the continuity of the flux  of $\psi$ at  each node, for all $t>0$,
\be\label{contpsi}\suEn D_i \partial_x\psi (N_\nu,t)=\suUn D_i \psix(N_\nu,t)\\ , \ee
and the conditions (\ref{TC})
ensure the conservation of the 
flux of the density of cells at each node $N_\nu$ , for $t>0$,
\be\label{consv}
\suEn \lai\v(N_\nu,t)=\suUn\lai\v(N_\nu,t)\ ,
\ee
which corresponds to the following condition for the evolution in time of the total mass
$$\suM\iIi\u(x,t)\ dx=\suM\iIi\uo(x)\ dx - \sum_{j\in\mathcal J} \int_0^t \mathcal W_j(s) ds\ .$$

\end{section}

\begin{section} 
{\bf  Local  solutions } 

In order to prove the existence and the uniqueness of a local solution to problem (\ref{sysi})-(\ref{compat}) we need to introduce some auxiliary functions.

We introduce the functions $\mathcal V(x,t)$ and $\Phi(x,t)$, defined on the network as follows
$$\left \{
\begin{array}{ll}\mathcal \dsp \mathcal V_i(x,t), \Phi_i(x,t)=0 \qquad \qquad\qquad { \t{ if }I_i    \t{ is an internal arc,}} \\ \\
\dsp\eta_j\lambda_{i(j)}\mathcal V_{i(j)}(x,t)= \frac{\mathcal W_j(t)}{L_{i(j)}}\left(\eta_jx +\frac {1-\eta_j}2 L_{i(j)}\right) &
 \t{ for all } j\in\mathcal J , \\  \\
\dsp \eta_jD_{i(j)}\Phi_{i(j)}(x,t)= \frac{\mathcal P_j(t)}{L^2_{i(j)}} x (x-L_{i(j)}) \left(x+ \frac {\eta_j-1}2L_{i(j)}\right) & \t{ for all } j\in\mathcal J \ ,\end{array} \right.$$ 
where $\eta_j$ is defined in the previous section.

Let  the triple $(u,v,\psi)$ be a solution to 
(\ref{sysi})-(\ref{compat}) and let 
\be\label{trasl} w := v-\mathcal V , \qquad   \phi =\psi-\Phi\ ;\ee
 the triple $(u,w , \phi)$ satisfies  the following system
\begin{equation}\label{ss}
    \left\{ \begin{array}{l}
\partial_t \u +\lambda_i \partial_x w_i= -\li \partial_x\mathcal V_{i}  
\\ \\
\partial_t w_i +\lambda_i \partial_x \u= \u (\partial_x \phi_i+\partial_x \Phi) -\beta_i w_i -\partial_t\mathcal V_i -\beta_i\mathcal V_i
\\ \\
\partial_t \phi_i= D_i \partial_{xx} \phi_i +a_i\u -b_i \phi_i -\partial_t\Phi_i +D_i \partial_{xx} \Phi_i -b_i \Phi_i\ 
 , 
\end{array}\right.\end{equation}
for $x\in I_i$, $i\in\M$,  $t>0$, with the initial conditions
\be\label{iiic}( u_i(x,0),  w_i(x,0),  \phi_i(x,0))=(u_{i0}(x),w_{i0}(x),\phi_{i0}(x))\ ,\ee
where $w_{i0}(x):= v_{i0}(x)-\mathcal V_i(x,0)  ,\phi_{i0}(x):=\psi_{i0}(x)-\Phi_i(x,0)\ ,$
the boundary conditions
\be\label{vbch}
\eta_j\lambda_{i(j)} w_{i(j)}(e_j, t)=0 ,
\qquad 
\  t>0\ , \ j\in\mathcal J\  ,
\ee
\be\label{phibch}
\eta_j D_{i(j)} {\partial_x \phi_{i(j)}}(e_j, t) + d_j \phi_{i(j)}(e_j,t)
=0
 \ ,\    	 d_j\geq 0\ ,   \  t>0 \ ,\ j\in \mathcal J\ , \ee
and  transmission conditions 
\be\label{KC2}\displaystyle
\delta_i^\nu \Di \partial_x \phi_{i}(N_\nu,t)=  \sum_{j\in\M^\nu} \alpha_{ij}^\nu(\phi_j(N_\nu,t)-\phi_i(N_\nu,t))\  \ ,\ i\in \M^\nu \ ,\  t> 0\ ,\ee
\be\label{TC2}\dsp
- \delta_i^\nu \lai w_i(N_\nu,t)=\sum_{j\in\M^\nu}
\sigma_{ij}^\nu \left(u_j(N_\nu,t)-\u(N_\nu,t)\right)\ ,\ i\in \M^\nu\ ,\ t>0\ , \ee
where $ \alpha_{ij}^\nu$ and $\sigma_{ij}^\nu$ are as in (\ref{KC}) and (\ref{TC}).

We are going to prove an existence and uniqueness result for local solutions to the above problem; as a consequence we will obtain the result for problem  (\ref{sysi})-(\ref{compat}).
 
  Let  $ X:=(L^2(\A))^2\ $ and $ Y:= (H^1(\A))^2\ $;
we consider the unbounded operator  $A_1:D(A_1)\to X$:
$$ \begin{array}{ll}
\dsp D(A_1)= \{\mathcal U=(u,w)\in Y
: (\ref{vbch}),(\ref{TC2}) \t{ hold }\}\\ \\
 A_1 \mathcal U=\{(-\lambda_i w_{ix}, -\lambda_i u_{ix})\}_{i\in \M}\ ;
\end{array}$$
moreover we introduce the  unbounded operator
$A_2:D(A_2)\to L^2(\A)$,
\be\label {da2} \begin{array}{ll}
D(A_2)= \left\{\phi\in H^2 (\mathcal A): (\ref{phibch}), (\ref{KC2}) \t{ hold }\right\}\\  \\
A_2\phi=\{D_i\phi_{ixx}-b_i\phi_i\}_{i\in\M}\ .
\end{array}\ee

\bp\label{mdiss}
$A_1$ and $A_2$  are m-dissipative operators.
\end{prop}

\begin{proof}

The proof for the operator
$A_1$ can be achieved as in \cite{noi}    (see the proof of Proposition 4.2), taking into account that the transmission conditions imply 
\be\label{444}\begin{array}{ll}
\dsp \sum_{\nu\in\mathcal N} \sum_{i\in\M^\nu} \delta_i^\nu\li u_i(N_\nu) w_i(N_\nu)=
\dsp  \sum_{\nu\in \mathcal N} \sum_{i,j\in\M^\nu}\frac{ \sigma^\nu_{ij}} 2\left( u_j(N_\nu,t)-\u(N_\nu,t)\right)^2
.  \end{array}\ee

For the operator  $A_2$ we notice that 
the transmission conditions (\ref{KC2}) imply that 
\be \label{coer}\begin{array}{ll}\dsp\sum_{\nu\in\mathcal N} \sum_{i\in\M^\nu} \delta_i^\nu D_i \phi_{i_x} (N_\nu)\phii(N_\nu)
=
\dsp-\sum_{\nu\in\mathcal N}\sum_{ij\in \M^\nu} \frac{\alpha^\nu_{ij}}{2} (\phi_j(N^\nu,t)-\phii(N^\nu,t))^2 ,
\end{array} \ee
while the homogeneous Robin boundary conditions provide the equality
$$\sum_{j\in\mathcal J} \eta _jD_{i(j)} \partial_x\phi_{i(j)} (e_j,t) \phi_{i(j)}(e_j,t)= -\sum_{j\in\mathcal J} d_j \phi^2_{i(j)}(e_j,t)\ \ , j\in\mathcal J .$$
Then by standard methods, $A_2$ reveals to be a dissipative operator \cite{CH}.
In order to prove that the operator in m-dissipative, we introduce the bilinear form $a(\phi,\chi):(H^1(\mathcal A))^2 \to \R$
$$ \begin{array} {ll}\dsp
a(\phi,\chi)=\suM\iIi \left(D_i\phx\chi_{ix} +(1+b_i) \phii\chi_i\right)\ dx
\\ \\ \dsp
-\sum_{\nu\in \mathcal N}\sum_{i,j\in \M^\nu}\alpha^\nu_{ij}\left(\phi_j(N_\nu) -\phii(N_\nu)\right) \chi_i(N_\nu)
+\sum_{j\in\mathcal J} d_j \phi_{i(j)}(e_j)\chi_{i(j)}(e_j)\ ;
\end{array}$$
the form is continuous and coercive, hence,
by the Lax-Milgram theorem, we know  that, for each $\varphi\in L^2(\mathcal A)$, there exists a unique
$\phi\in H^1(\mathcal A)$ such that, for all $\chi\in H^1(\mathcal A)$ it holds
$$a(\phi,\chi)=\suM\iIi \varphi_i\chi_i\ dx ;$$
taking $\chi_i\in H^1_0(I_i)$ for all $i\in \M$, we obtain that 
$\phx\in H^1(I_i)$, then
taking $\chi_i\in C^\infty_0(I_i)$, as in \cite{noi}, we prove the equality 
$$-D_i \phxx +(1+b) \phii =\varphi_i\quad a.e.\t{ for all }i\in \M, $$
 moreover, thanks to suitable choices of $\chi_i(N), \chi_i(a_i)$,
we obtain that $\phi$ satisfies the right boundary and transmission conditions to belong to $D(A_2)$.
\end{proof}

Thanks to the above proposition we conclude that the operator $A_1$   is the generator of a contraction semigroup $\T_1$
in $(L^2(\A))^2$ while  the operator $A_2$   is the generator of a contraction semigroup $\T_2$
in $L^2(\A)$.

It is easy to prove a uniqueness result for solutions to the problem (\ref{ss})-(\ref{TC2}):
if  we assume that $(u,w,\phi),(\ov u,\ov w,\ov\phi)$ are two solutions such that  $\phi,\ov \phi\in C([0,T];H^2(\A))\cap C^1([0,T];L^2(\A))$ and 
$(u,w),(\ov u,\ov w) \in  C([0,T];Y)\cap C^1([0,T];X)$,
taking into account (\ref{444}) and (\ref{coer}), by standard methods, we have for $t\in [0,T]$,
$$
\Vert u(t)-\ov u(t)\Vert_2^2 +\Vert w(t)-\ov w(t)\Vert_2^2 +\iot \Vert w(s)-\ov w(s)\Vert_2^2  ds $$
$$ \leq c_1  \iot\left(\Vert \phi_{x}(s)-\ov \phi_{x}(s)\Vert_2^2 +\Vert u(s)-\ov u(s)\Vert_2^2 \right) ds$$
 and
 $$
\Vert \phi(t)-\ov \phi(t)\Vert_2^2 +\iot\left( \Vert\phi(s)-\ov \phi(s)\Vert_2^2+ \Vert\phi_{x}(s)-\ov \phi_{x}(s)\Vert_2^2\right)  ds
\leq 
c_2\iot \Vert u(s) -\ov u(s)\Vert_2^2 ds\ ,$$
where the constants $c_1, c_2$ depend on $\dsp\soT\Vert u(t)\Vert_{H^1}, \soT \Vert\ov  \phi_x(t)\Vert_{H^1}$, on $\Phi$ and on the parameters $a_i,b_i, D_i, \li, \beta_i$; then the result follows by Gronwall Lemma.

In order to prove the local existence theorem we need some preliminary results.
Let $f\in C([0,T]; H^2(\A)) \cap H^1((0,T);H^1(\A)) $ and   $g=(g_1,g_2)\in C([0,T];Y)\cap C^1([0,T];X)$, we set
\be\label{Ffg}F_{fg}(t)=\{ (0\ ,\ { f_i}_x(t)g_{1i}(t)+\Phi_{ix}(t)g_{1i}(t)-\beta_i g_{2i}(t) )\}_{i\in\M}	.\ee
\begin{lemma}
\label{co} Let  $\ov g,g\in C([0,T];Y)\cap C^1([0,T];X)$, 
$\ov f,f\in C([0,T]; H^2(\A)) \cap H^1((0,T);H^1(\A))$
and 
$$\dsp \soT\Vert  f(t)\Vert_{H^2}+\left( \ioT \Vert f'_x(t)\Vert_2^2 dt\right)^{\frac 1 2} \ ,\ 
\soT\Vert  \ov f(t)\Vert_{H^2}+\left( \ioT \Vert \ov f'_x(t)\Vert_2^2 dt\right)^{\frac 1 2}
  \leq K .$$
 Then there exist two positive constants
$L_{1K}, L_{2K}$, depending  on  $K$ 
, such that 
$$\dsp\soT\Vert F_{fg}(t)- F_{\ov f\ov g}(t)
\Vert_X \leq 
 L_{1K}\dsp \soT \Vert g(t)-\ov g(t)\Vert_X \\ \\
  +\dsp \soT \Vert \ov g_1(t)\Vert_{\infty}   \dsp\soT \Vert f_x(t)-\ov f_x(t)\Vert_2\ ,$$
and 
\be\label{co2}  \begin{array}{ll} \dsp \int_0^T \Vert  F'_{fg}(t)-  F'_{\ov f\ov g}(t)\Vert_X dt
\leq \sqrt T L_{2K} \left(\soT \Vert g(t)-\ov g(t)\Vert_Y +\sqrt T \Vert g'(t)-\ov g'(t)\Vert_X\right) \\ \\
 +\dsp \sqrt T \soT\Vert \ov g_1(t)\Vert_\infty \left(\ioT \Vert f_x'(t)-\ov f_x'(t)\Vert_2^2 dt\right)^{\frac 1 2 }
+ \dsp T \soT \Vert \ov g'_1(t)\Vert_2  \soT \Vert f_x(t)-\ov f_x(t)\Vert_{\infty}.
\end{array}\ee
\end{lemma}
\begin{proof}
We have
$$\begin{array}{ll}\dsp\soT\Vert F_{fg}(t)- F_{\ov f\ov g}(t)
\Vert_X \leq  \left(\dsp \soT \Vert f_x(t)\Vert_\infty+
\dsp\soT \Vert \Phi_x(t)\Vert_{\infty}\right) \soT \Vert g_1(t)-\ov g_1(t)\Vert_2\\ \\
+\ov\beta  \dsp\soT \Vert g_2(t)-\ov g_2(t)\Vert_2  +
\dsp  \soT \Vert \ov g_1(t)\Vert_{\infty} \dsp \soT \Vert f_x(t)-\ov f_x(t)\Vert_2 \ ,
\end{array}$$
where $\ov \beta:=\max\{\beta_i\}_{i\in\M}$; then the first inequality in the claim follows with  
$ L_{1K}=    c_S K  +\dsp\soT\Vert \Phi_x\Vert_\infty +\ov \beta $, where $c_S$ is a Sobolev constant.

As regard to the second inequality we have
$$\int_0^T \Vert  F'_{fg}(t)-   F'_{\ov f\ov g}(t)\Vert_X \leq 
  \left( \ioT \Vert g_1(t)-\ov g_1(t)\Vert^2_{\infty} dt\right)^{\frac 1 2} \left(\ioT \Vert f_x'(t)+ \Phi'_{x}(t)\Vert_2^2 dt\right)^{\frac 1 2 } 
$$
$$
+ \left( \ioT \Vert \ov g_1(t)\Vert^2_{\infty} dt\right)^{\frac 1 2} \left(\ioT  \Vert f_x'(t)-\ov f'_x(t)\Vert_2^2 dt \right)^{\frac 1 2 } 
+T \ov \beta \soT \Vert g_2'(t)-\ov g_2'(t)\Vert_2
$$ $$
 + T\left(  \soT \Vert f_x(t)\Vert_{\infty} + \soT\Vert\Phi_x(t)\Vert_{\infty}\right) \soT \Vert g_1'(t)-\ov g_1'(t)\Vert_2 + 
 $$ 
 $$
  +T 
\soT \Vert \ov g'_1(t)\Vert_2  \soT \Vert f_x(t)-\ov f_x(t)\Vert_{\infty}
$$
$$\leq
 \sqrt T c_S \left(K+\Vert\Phi \Vert_{H^1((0,T);H^1)}\right)\soT \Vert g(t)-\ov g(t)\Vert_Y
$$
$$ +T  (c_S (K+ \dsp\soT\Vert\Phi_x(t)\Vert_{H^1})+\ov \beta)  \soT \Vert g'(t)-\ov g'(t)\Vert_X$$ $$
+\sqrt T \dsp \soT\Vert \ov g_1(t)\Vert_\infty \left(\ioT \Vert f_x'(t)-\ov f_x'(t)\Vert_2^2\right)^{\frac 1 2 }
+ T \soT \Vert \ov g'_1(t)\Vert_2  \soT\Vert f_x(t)-\ov f_x(t)\Vert_{\infty}\ ,
$$
where $c_S$ are Sobolev constants; then, setting $ L_{2K}=    c_S (K  +c_\Phi)+\ov\beta$,  where $c_\Phi$ 
is a suitable positive quantity depending on $\Phi$,
we obtain the second inequality.

\end{proof}

\bt \label{le}{\it (Local existence)}
There  exists a unique local solution $(u,v,\psi)$ to problem (\ref{sysi})-(\ref{compat}),
$$\begin{array}{ll}
(u,v)\in C([0,T];(H^1(\mathcal A))^2)\cap C^1([0,T],(L^2(\mathcal A))^2) ,\\ \\ 
\psi \in C([0,T];H^2(\mathcal A))\cap C^1([0,T],L^2(\mathcal A))\cap H^1((0,T);H^1(\mathcal A))
\ .
\end{array}$$
\et
\begin{proof}
We set 
$\ov a:=\dsp \max_{i\in\M} a_i $, $ \underline b:=\dsp\min_{i\in\M} b_i $, $\overline b:=\dsp\max _{i\in\M} b_i $ and $\underline D:=\dsp\min_{i\in\M} D_i$.

We consider the problem (\ref{ss})-(\ref{TC2}) and we set $\U_0:=(u_0,w_0)$
and 
$$Z_i(t)=(Z_{1i}(t),Z_{2i}(t))):=(-\lambda_i \mathcal V_{ix}, 	-\mathcal V_{it} -\beta_i \mathcal V_i)\ ,$$ 
$$Z_{3i}(t):=  -\partial_t\Phi_i +D_i \partial_{xx} \Phi_i -b_i \Phi_i.  $$
Fixed $\ov T>0$, we set 
$$M\geq 2\left(\left( 1+\Vert {\phi_0}_x \Vert_{\infty}+\Vert\Phi_x\Vert_\infty+\ov\beta \right) \Vert \U_0\Vert _{D(A_1)} 
+\Vert Z(0)\Vert_X
+  \Vert Z\Vert_{W^{1,1}((0,\ov T); X)}\right) ,$$
$$K_1= \Vert \phi_0\Vert_{D(A_2)} +\ov T3\ov aM +2 \Vert Z_3\Vert _{W^{1,1}((0,\ov T);L^2(\A))}  +2(\ov a\Vert u_0\Vert_2+
\Vert Z_3(0)\Vert_2),$$
 $$K_2=\frac{1}{2\underline D}\left( \Vert \phi_0\Vert_{D(A_2)}+\ov a \Vert  u_0\Vert_2+ \Vert Z_3(0)\Vert_2\right)^2
 +  \frac {1} {2\underline b\,  \underline D}\left(  \ov T\ov a^2 M^2+
  \Vert {Z_3}\Vert^2_{H^1(0,\ov T);L^2(\A)) }\right) ,$$
$$K\geq \left (1 +\frac {1+\ov b}{\underline D} \right)K_1+\sqrt K_2 \ .
 $$
Let $L_{1K}, L_{2K}$ be the constants in Lemma \ref{co} and let $T\leq\ov T$.

Let consider the set
\be \label{bmk}B_{MK}=\left\{
\begin{array} {ll}\dsp
\U=(u,w)\in  (C([0,T];(H^1(\mathcal A))^2) \cap C^1([0,T];(L^2(\A))^2)\ , \\ \\ \phi\in C([0,T];H^2(\A))\cap C^1([0,T];L^2(\A))\cap H^1((0,T);H^1(\A)) \ :\\ \\
\U(0)=(u_0,w_0) ,\phi(0)=\phi_0\ , \\ \\
  \dsp \soT\Vert(\U(t)\Vert_X ,\ \ \   \soT\Vert(\U'(t))\Vert_X\leq M\ ,\\ \\ 
  \dsp \soT \Vert A_1 \mathcal U(t)\Vert_X \leq (1+L_{1K})M+\soT \Vert Z(t)\Vert_X\ ,
\\ \\ \dsp \soT\Vert \phi(t)\Vert_{H^2}\leq   \left(1+\frac {1+\ov b}{\underline D}\right)K_1,\ \ \ 
  \dsp
 \int_0^T \Vert\phi'_{x}(t) \Vert_2^2 \, dt   \leq   K_2
 \end{array}
\right\}\ ;\ee
we equip $B_{MK}$ with the metric generated by the norms of the involved spaces, obtaining a  complete metric space .

We  define a map $G$ on $B_{MK}$ in the following way:
given $(\mathcal U^I, \phi^I)=(u^I,w^I,\phi^I)\in B_{MK}$ , then $(\mathcal U,\phi)= G(\mathcal U^I, \phi^I)$ 
is such that 
$\phi$ is the solution to
\be\label{A2} \left\{\begin{array}{ll}
 \phi \in C([0,T];D(A_2))\cap C^1([0,T];L^2(\A)) \\ \\
\phi'(t) =A_2 \phi(t) + a u^I(t) +Z_3(t)\ ,\quad t\in[0,T]\ ,\\ \\
\phi(0)=\phi_0\in D(A_2)\ ,
\eaa\right.\ee
where $au^I(t):=\{a_i u^I_i(t)\}_{i\in\M}$,  and $\U$ is the solution to 
\be\label{A1} \left\{\begin{array}{ll}
\mathcal U \in C([0,T];D(A_1))\cap C^1([0,T];X) \\ \\
\mathcal U'(t) =A_1 \mathcal U(t) +F_{\phi \U^I}(t) + Z(t)\ ,\quad t\in[0,T]\ ,\\ \\
\mathcal \U(0)=(u_0,w_0)\in D(A_1)\ ,
\eaa\right.\ee
where we used the notation  (\ref{Ffg}).

First we prove that  $G$ is well defined and $G(B_{MK})\subseteq B_{MK}$.
Since  $\U^I\in C([0,T];Y) \cap C^1([0,T];X)$, 
$ Z_3\in H^{1}((0,T);L^2(\A))$ 
we can use the theory for nonhomogeneous problems in \cite{CH} and we infer the existence and uniqueness of a solution
 $\phi$ to problem  (\ref{A2})
given by
$$\phi(t) =\T_2(t) \phi_0 +\iot \T_2(t-s) (au^I(s)+Z_3(s)) ds\ ,$$
see\cite{CH}. If we set 
$$\dsp\mathcal F(t):=\iot \T_2 (t-s) (au^I(s)+Z_3(s)) \ ds\ ,
\ $$ 
the assumption on $u^I$ and $Z_3$ imply that 
$\mathcal F \in C^{1}([0,T]; L^2(\mathcal A))\cap  C([0,T]; D(A_2))$,
$A_2 \mathcal F(t)=\mathcal F'(t)-au^I(t)-Z_3(t)$ for all $t\in [0,T]$ (see \cite{CH})
and
$$ \mathcal F'(t)= \iot  \T_2(t-s)(a{u^I}'(s)+Z'_3(s)) \ ds + \mathcal T_2(t)(a u^I(0)+Z_3(0)) \ .$$

Then we have
$$
\Vert \phi(t)\Vert_{D(A_2)} \leq \Vert \phi_0\Vert_{D(A_2)} +
\Vert \mathcal F(t)\Vert_{2}+\Vert A_2\mathcal  ( \mathcal F(t) \Vert_{2}$$
$$\leq  \Vert \phi_0\Vert_{D(A_2)}+\ov a \Vert u^I(t)-\mathcal T_2(t)u^I(0)\Vert_2 +  \Vert Z_3(t)-\mathcal T_2(t)Z_3(0)\Vert_2$$
$$+
 \iot  \left( \ov a\Vert u^I(s)\Vert_2+  \Vert  Z_3(s)\Vert_2+ \ov a\Vert ({u^I}'(s)\Vert_2+ \Vert  Z'_3(s)\Vert_2\right) ds
$$
$$\leq     \Vert \phi_0\Vert_{D(A_2)}
+\ov aT\left(  2\soT    \Vert {u^I}'(s)\Vert_2
+  \soT\Vert u^I(s)\Vert_2\right) $$
$$+ 2\ov a \Vert u(0)\Vert_2+ 2\Vert Z_3(0)\Vert_2+2 \Vert Z_3\Vert _{W^{11}(0,T);L^2(\A))}\leq K_1\ ,$$
whence the first inequality for $\phi$ in $B_{MK}$ follows.

Moreover, let $0<t_1<t_2<T$ and 
let $\Delta^h \chi(t):= \chi(t+h)-\chi(t)$; using the equation in (\ref{A2}) we can write
$$\int_{t_1}^{t_2}\iIi\left((\Delta^h \phii)_t \Delta^h \phii - D_i (\Delta^h \phii)_{xx}\Delta^h\phii 
\right) dx dt=$$
$$\int_{t_1}^{t_2}\iIi\left(
a_i\Delta^h u^I_i\Delta^h \phii-b_i(\Delta^h \phii)^2 +\Delta^hZ_3 \Delta^h\phi  \right) \, dx dt \ ;$$
then we have
\be\label{uuu}\begin{array}{ll}\dsp
\suM \left(\iIi \frac{(\Delta^h \phii(t_2))^2} 2\, dx+
\int_{t_1}^{t_2}\iIi D_i(\Delta^h\phx)^2\ dx dt\right) \\ \\ \leq  \dsp
 \sum_{\nu\in\mathcal N}\sum_{i\in\M^\nu}  \int_{t_1}^{t_2}\delta_\nu^i D_i(\Delta^h \phx) (\Delta^h \phii) (N_\nu,t)  dt \\ \\
+\dsp \sum_{j\in\mathcal J} \int_{t_1}^{t_2} \eta_j D_{i(j)} (\Delta^h \phi_{i(j)_x}) (\Delta^h \phi_{i(j)}) (e_j,t)  dt\\ \\
\dsp+ \suM \iIi \frac {(\Delta^h \phii(t_1))^2} 2 dx +\suM
\int_{t_1}^{t_2}\iIi \left(  \frac{(a_i\Delta^h u^I_i)^2 +(\Delta^h Z_{3i})^2}{2b_i}  
\right)dx dt  \ .
\end{array}\ee

Since the first and the second terms on the
right hand side are  non positive and  $u^I,\phi\in C^1([0,T]; L^2(\A))$ and $Z_3\in H^1(0,T);L^2(\A))$, the above inequality implies that 
$\phi\in H^1((0,T); H^1(\A))$; moreover, 
for $h\to 0$ and then  $t_1\to 0$, $t_2\to T$,  we have
$$\suM\int_0^T\iIi D_i {\phi_i}_{xt}^2 \ dx dt\leq 
  \frac 1 2 \left(\Vert \phi_0\Vert _{D(A_2)} +  \ov a  \Vert  u_0\Vert_2 +   \Vert  Z_3(0)\Vert_2 \right)^2$$
$$\qquad+ \frac {1} {2 \underline b} \left(T\ov a^2 \soT   \Vert  {u^I}'(t)\Vert_2^2+
  \Vert {Z_3}\Vert_{H^1((0,T);L^2(\A))}^2\right)\ ,
$$
so that $\phi$ satisfies the last condition in (\ref{bmk}).

Now we consider the problem (\ref{A1}) and we set
$F(t):= F_{\phi \U^I}(t)$. We know that $Z\in  W^{1,1}((0,T);X)$ and, from Lemma \ref{co}, 
 $F\in W^{1,1}((0,T);X)$, then there exists a unique solution
$\U\in $ given by
$$\U(t) =\T_1(t) \mathcal U_0 +\iot \T_1(t-s) (F(s) +Z(s)) ds$$
see \cite{CH};
moreover, using  Lemma \ref{co} and choosing $T\leq(2L_{1K})^{-1}$,
we obtain the following inequality
$$\Vert\U(t)\Vert_X \leq \Vert \mathcal U_0\Vert_X +\dsp\iot \Vert F(s) + Z(s)\Vert_X ds $$
$$ \leq \Vert \mathcal U_0\Vert_X +T L_{1K} \soT\Vert \U^I(t)\Vert_X +  \Vert Z\Vert_{L^1((0,T);X)} \leq M\ ;
$$
then we can argue as we did before for $\phi$, using \cite{CH} and 
   Lemma \ref{co},  to obtain  
\be\label{ut}\begin{array}{ll}\Vert \U'(t)\Vert_X\leq \Vert A_1 \mathcal U_0\Vert_X + \Vert F(0)+Z(0)\Vert_X+\dsp \iot  \Vert F'(s) +Z'(s)\Vert_X ds\\ \\
\dsp \leq \Vert A_1 \mathcal U_0\Vert_X  + (\Vert {\phi}_x(0)\Vert_\infty+\Vert\Phi_x(0)\Vert_\infty+
\ov \beta) \Vert \mathcal U_0\Vert_X + \Vert Z(0)\Vert_X \\ \\
+\sqrt T L_{2K} \left(\dsp\soT \Vert u^I(t)\Vert_Y +\sqrt T\dsp \soT \Vert {\U^{I'}(t)}\Vert_X\right) +   \Vert Z\Vert_{W^{1,1}((0,T);X)} 
\end{array}\ee
which implies, choosing $T$ sufficiently small and setting $\underline \lambda:=\dsp\min_{i\in\M} \li$,
 $$\soT \Vert \U'(t)\Vert_X \leq \frac M 2 + \sqrt T L_{2K}\left( \left(1+\frac{1+L_{1K}}{\underline \lambda}  +\sqrt T \right)M +
 \frac{\dsp \soT \Vert Z(t)\Vert_X}{\underline \lambda} \right)\leq M\ .$$
 Finally, using Lemma \ref{co},
 $$ \soT \Vert A_1 \U(t)\Vert_X \leq\soT \Vert \U'(t)\Vert_X+\soT\Vert  F(t)+Z(t) \Vert _X 
  \leq (1+L_{1K}) M +\soT \Vert Z(t)\Vert_X.$$
 
 The above computations show that $(\U,\phi)=G(\U^I,\phi^I) \in B_{MK}$ if $T$ is small enough.

Now we are going to prove that $G$ is a contraction mapping on $B_{MK}$,  for small values of $T$.

Let 
$$(\U^I,\phi^I)= (u^I, w^I,\phi^I)\ ,\ (\ov\U^I,\ov\phi^I)=(\ov u^I,\ov w^I, \ov \phi^I)\in B_{MK}\ ,$$
$$(\U,\phi)=G( \U^I,\phi^I)\ ,\  (\ov\U,	\ov\phi)=G(\ov \U^I,\ov\phi^I)\ ,$$  
$$\ov F(t)=F_{\ov \phi\ov \U^I}(t)\ ,\ 
 F(t)=F_{\phi\U^I}(t)\ ;$$ 
Then we have, arguing as for the previous estimates for  $\phi$,
$$\soT \Vert \phi(t) -\ov \phi(t)\Vert_{D(A_2)} 
\leq T\ov a  \left(2\soT \Vert{ \U^I}'(t)-{\ov\U^I}'(t)\Vert_X +\soT \Vert \U^I(t)-\ov\U^I(t)\Vert_X
\right);$$
using (\ref{uuu})
$$\frac 1 2 
\soT \Vert \phi'(t)-\ov \phi'(t)\Vert_2^2+
\underline D \ioT \Vert {\phi}_{xt}(t)-{\ov\phi}_{xt}(t)\Vert_2^2 dt
 \leq \frac{T\ov a^2} {2\underline b} \soT \Vert { \U^I}'(t)-{\ov\U^I}'(t)\Vert_X^2\ ,$$
Then, using , Lemma \ref{co}, 
$$\begin{array}{ll}\dsp
\soT \Vert \ov {\mathcal U}(t)-\mathcal U(t)\Vert_X
\leq 
  \soT\left\Vert\iot \mathcal T_1(t-s)  (\ov F(s)- F(s) ) \ ds\right\Vert_X
\\ \\ \dsp
\leq T L_{1K} \soT \Vert \ov\U^I(t)- \U^I(t)\Vert_X\\ \\ +T \frac{c_S}{\underline \lambda}
\left((\underline \lambda+1+L_{1K})M +\dsp\soT \Vert Z(t)\Vert_X\right)\dsp \soT\Vert \ov\phi_x(t)-\phi_x(t)\Vert_2\ ,
\end{array}\ $$
where $c_S$ is a Sobolev constant; moreover,  arguing as in (\ref{ut}) and using (\ref{co2})
$$
\sup_{[0,T]}\Vert \ov \U'(t)- \U'(t)\Vert_X\leq  \ioT \Vert F'(t)-\ov F'(t)
\Vert_X  dt
\  \leq \  \sqrt  T L_{2K} 
 \Vert \ov\U^I-\U^I\Vert_{C([0,T];Y)} $$
 $$+ \sqrt T\frac{c_S}{\underline \lambda}
 \left(	(\underline \lambda +1+L_{1K})M +\Vert Z\Vert _{C[0,T];X)}\right)
  \Vert \phi_x -\ov \phi_x\Vert_{H^1((0,T);L^2)}
$$
$$ + T\left(L_{2K} \Vert \ov\U^I-\U^I\Vert _{C^1[0,T]; X)} + Mc_S \Vert \phi -\ov \phi\Vert_{C([0,T];H^2)}\right)\ .
$$
Finally, for $t\leq T$,
$$
\Vert A_1(\U(t))-A_1(\ov \U(t))\Vert_X \leq  \Vert \U'(t)-\ov \U'(t)\Vert_X+ \Vert F(t)-\ov F(t)\Vert_X  
\leq  \Vert \U'(t)-\ov \U'(t)\Vert_X
$$ $$
+L_{1K} T \Vert {\ov\U^I}'- {\U^I}'\Vert_{C([0,T];X)} + \frac{c_S}{\underline \lambda} 
\left((\underline \lambda+1+L_{1K})M+\Vert Z\Vert_{C([0,T];X)}\right)
\Vert \ov\phi_x(t)-\phi_x(t)\Vert_2 .$$
Hence, if  $T$ is sufficiently small, then $G$ is a contraction mapping in $B_{MK}$ and
the unique fixed point  $(\mathcal U,\phi)\in B_{MK}$ is the solution to problem (\ref{ss})-(\ref{TC2}).
The existence and uniqueness of local solutions  for problem (\ref{sysi})-(\ref{compat}) follow from  (\ref{trasl}).
\end{proof}

\end{section}

\begin{section} 
{\bf A priori estimates}

In this section we assume the condition (\ref{nd}) .
Moreover,we assume that there exists a stationary solution
$(U(x),V(x),\Psi(x))$ to problem (\ref{sysi}),( \ref{KC}), (\ref{TC}), 
 verifying the boundary conditions
\be\label{ub}\eta_j \lambda_{i(j)} V(e_j) =W_j \ , \ \eta_jD_{i(j)} \Psi_{i(j)}(e_j) + d_j \Psi_{i(j)} (e_j) =P_j\  , \ j\in\mathcal J\ ;\ee
we notice that, integrating the first equation in (\ref{sysi}) and using  the conservation of the flux (\ref{consv})  at each inner 
node, it turns out to be 
 necessary that 
$\dsp \sum_{j\in \mathcal J}W_j =0 \ $ . 

We set 
\be\dsp\label{um}\mu_s=\suM \iIi U_i(x) dx \ . \ee
Let 
$(\ov u(x,t), \ov v(x,t),\ov \psi(x,t))$
 be the local solution  in Theorem \ref{le} to problem (\ref{sysi})-(\ref{compat}); 
we denote with   $(\ov u_0(x),\ov v_0(x), \ov\psi_0(x)) $ the initial data, with  $\dsp\ov \mu(0):=\suM\iIi \ov u_0(x) dx $ 
the initial mass and with $ \mathcal {\ov W}(t),  \mathcal {\ov P}(t)$ the boundary data, so that, integrating the first equation in (\ref{sysi}), we have the following expression for the mass at the time $t$
 $$\ov \mu(t) =\ov  \mu(0)-\sum_{j\in\mathcal J}  \int_0^t \mathcal {\ov W}_j(s)ds\ .$$
  Due to the assumption of existence of the  stationary solution $(U(x),V(x),\Psi(x))$,
 the triple $$(u(x,t),v(x,t),\psi(x,t))
=(\ov u(x,t)-U(x),\ov v(x,t) -V(x), \ov \psi(x,t)-\Psi(x))$$ is the local  solution of the following problem
\begin{equation}\label{sysas}
    \left\{ \begin{array}{l}
\partial_t \u +\lambda_i \partial_x \v=0\ ,
\\ \\
\partial_t \v +\lambda_i \partial_x \u= \u \partial_x \psi_i -\beta_i \v + u_i \partial_x \Psi_i  +U_i \partial_x \psi_i \ ,
\ \      t\in[0,T]  , \ x\in I_i  , \  i\in \M,
\\ \\
\partial_t \psii= D_i \partial_{xx} \psi_i +a_i\u -b_i\psi_i\\ \\
u(x,0)=u_0(x):=\ov u_0(x)-U(x), \ v(x,0)=v_0(x):=\ov v_0(x)-V(x),\\ \\ 
 \psi(x,0)=\psi_0(x):=\ov \psi_0(x)-\Psi(x) \ ,\\ \\
\eta_j\lambda_{i(j)} v_{i(j)}(e_j, t)=\mathcal W_j(t):=\mathcal{\ov W}_j (t)-W_j \quad\qquad\qquad \qquad\qquad j\in\mathcal J\ , \\ \\ 
\eta_j D_{i(j)} {\psi_{i(j)}}_x(e_j, t) + d_j \psi_{i(j)}
=\mathcal  P_j(t) :=\mathcal{\ov P}_j(t)-P_j\ ,\qquad\qquad j\in\mathcal J\ ,
\end{array}\right.\end{equation}
complemented with the transmission conditions (\ref{TC}) and (\ref{KC}). We set 
\be\label{uum} \mu(t):=\suM\iIi u_i(x,t) dx =\ov \mu(t)- \mu_s\ .\ee
 We are going to prove  some a priori estimates for the solution to the above problem, assuming suitable conditions on the data.
If the stationary solution and the data in (\ref{sysas}) are small in some suitable norms, then these estimates provide a global existence result for problem (\ref{sysas}). In this way,  after the proof of the real existence of stationary solutions in the next section,  the results in the following propositions will be the tools to prove the existence of global solution to problem (\ref{sysi})-(\ref{compat}).

We assume the following relations among the initial and boundary data of $(\ov u, \ov v,\ov \psi)$ and the stationary solution:
\be\label{WW} \mathcal{\ov W}_j-W_j\in    W^{2,1}((0,+\infty)) \ ,\ \t{ for all } j\in \mathcal J\ ,\ee
\be\label{PP}\mathcal {\ov P}_j\in H^2((0,T))\ \t{ for all } T>0\  , 
\mathcal {\ov P}_j-P_j \in  H^{1}((0,+\infty)) \ ,\   \t{ for all }  j\in \mathcal J\ ,\ \ee
\be\label{mut}\ov  \mu(t) -\mu_s \in L^2((0,+\infty))\ .\ee
First we remark that the assumption (\ref{nd}) implies that 
the condition (\ref{TC})
can be rewritten as follows
\be\label{ccnd}
\dsp u_j(N_\nu,t)=u_{k_\nu} (N_\nu ,t)+\sum_{i\in\M^\nu} \gamma^\nu_{ij} \v(N_\nu,t)\qquad \t{ for all } j\in\M^\nu ,
 \nu\in\mathcal N\ ,\ee
for suitable $\gamma^\nu_{ij}$ (see the proof of Lemma 5.9 in \cite{noi}). This equality allow to prove 
 the following  estimate for $ \Vert \u(\cdot,t)\Vert_\infty $. 

Set $|\A|:=\dsp\suM L_i$.

\bp\label{uinf} Let (\ref{nd}) hold;
let $u,v\in C([0,T];L^2(\A))\cap C^1([0,T];H^1(\A))$ satisfying the conditions (\ref{TC});
then, for 
 all $i\in\M$, $0\leq t\leq T$,
$$ \Vert u_i( \cdot,t)\Vert_\infty \leq \frac{|\mu(t)|} {|\A|}+2 \suMj\left( 2\Vert u_{jx}(\cdot,t)  \Vert_1+3 \gamma
 \Vert v_j(\cdot,t)\Vert_\infty \right)\ ,$$
where $\mu(t)=\dsp \suM u_i(x,t) dx$ and $\gamma =\max\{ \vert \gamma_{ij}^\nu\vert\}$ .

\epr

\begin{proof}

We consider two consecutive nodes, $N_\nu$ and $N_h$,  and let $I_l$ be the arc linking them.
For all $x\in I_l$, $t\in[0,T]$
$$u_l(x,t)= u_l(N_\nu,t)+\int_{N_\nu}^x u_{l_y} (y,t) dy= u_l(N_h,t)+\int_{N_h}^x u_{l_y}(y,t)dy\ $$
(by $N_\nu$ we mean $0$ if $N_\nu$ is the starting  node of $I_l$ and we mean $L_l$ otherwise);
 let  $k_\nu,k_h$ the indexes relative to the nodes, $N_\nu$ and $N_h$ in condition (\ref{nd}), then, using (\ref{ccnd}), we can write
  for all $t\in[0,T]$,
$$\begin{array}{ll} \dsp
u_{k_\nu} ( N_{\nu},t) +\dsp \sum_{j\in\M^\nu} \gamma^{\nu}_{l j} v_j(N_{\nu},t) + \int_{N_\nu}^x u_{l_y} (y,t)dy\\ \\
\qquad\qquad =u_{k_h} ( N_{h},t) \dsp
+ \sum_{j\in\M^{h} }\gamma^{h}_{lj} v_j(N_{h},t)+ \int_{N_h}^x u_{l_y}(y,t) dy\ ;
\end{array}$$
then
$$
u_{k_\nu} ( N_{\nu},t) =u_{k_h} ( N_{h},t) -  \sum_{j\in\M^{\nu}} \gamma^{\nu}_{l j} v_j(N_{\nu},t)  + \sum_{j\in\M^{h} }\gamma^{h}_{lj} v_j(N_{h},t)+ \int_{N_h}^{N_\nu} u_{l_y}(y,t) dy .
$$
Since each node of the network is connected with the node  $N_1$, the above relation implies that, for all $p\in\mathcal N$, we can express the value of $u_{k_p}(N_p,t)$ in the following way
$$u_{k_p}(N_p,t)= u_{k_1}(N_1,t) +  \Gamma_p (t)\ ,$$
where  $k_p$ and $k_1$ are the indexes in condition (\ref{nd}) relative to $N_p$ and $N_1$ respectively, and
\begin{equation}\label{Gamma}\dsp\vert \Gamma_p(t)\vert \leq  \suMj \left( 2 \gamma \Vert v_j(t)\Vert_\infty+ \Vert u_{j_x}(t)\Vert_1\right) \ .\ee
For all $i\in \M_1$, thanks to condition (\ref{nd})), we have, for all $x\in \Ii$, $t\in[0,T]$,
$$ \u(x,t)= u_{k_1}(N_1,t) + \sum_{j\in\M_1} \gamma^1_{ij} v_j(N_1,t) +\int_{N_1}^x \ux (y,t) dy,
$$
and, thanks to the previous computations, a similar expression can be derived for all $i\in \M_p$, for all $p\in\mathcal N$:
$$\u(x,t)= u_{k_1}(N_1,t) + \Gamma_p(t)+ \sum_{j\in\M_p} \gamma^p_{ij} v_j(N_p,t) +\int_{N_p}^x \ux (y,t) dy\ \ \t{ for all } x\in\Ii.
$$
Obviolusly, each $u_i$ has not a unique expression; in all cases, for all $i\in\M$ we can write
\begin{equation} \label{svolta3}
\u(x,t)= u_{k_1}(N_1,t) +\ov \Gamma_{i}(t)+ \int_{N_q}^x \ux (y,t) dy\ \ \t{ for all } x\in\Ii.
\ee
where $N_q$ is one of the extreme points of $I_i$ and  $\ov \Gamma_{i}(t)$ is a suitable  quantity verifying
\be\label{svolta4}\vert \ov \Gamma_i(t) \vert \leq \suMj \left( \Vert u_{jx}(t)  \Vert_1+ 3 \gamma \Vert v_j(t)\Vert_\infty \right) \ .\ee
Integrating on $\Ii$ the  equality (\ref{svolta3}), we obtain
$$\dsp L_i u_{k_1}(N_1,t) =  \iIi  \u(x,t)dx- L_i \ov \Gamma_{i}(t)
-\dsp \iIi\int_{N_q}^x \ux (y,t) dy\ dtx\  ,$$
whence, summing for $i\in \M$ and using (\ref{svolta4}), we infer 
that, for all $t$,
$$\dsp |u_{k_1}(N_1,t)|\leq    \frac {|\mu(t)|} {|\A|}+ \suM \left(3\gamma  \Vert\v(t)\Vert_\infty +2\Vert \ux(t)\Vert_1 \right)\ .
$$
Now we use this inequality in (\ref{svolta3}) to obtain the claim.

\end{proof}

Now we are going to obtain  a priori estimates necessary to prove the uniform (in time) boundedness of some norms  of $(u,v,\psi)$ when the data are small. Similar results are proved in \cite{noi} in the case of homogeneous boundary conditions, when  $\frac {a_i}{b_i}$  does not change with $i$, and some of the proofs have minor  differences from the ones in that paper .
\bp\label{I}
Let (\ref{nd}) hold and 
let $(u,v,\psi)$ be the local solution to (\ref{sysas})-(\ref{mut}),(\ref{KC})-
(\ref{compat});
then
$$\begin{array}{ll}
\displaystyle
\suM\left(\soT \Vert\u(t)\nld^2+\soT\Vert\v(t)\nld^2 + 2\beti \ioT \Vert\v(t)\nld^2 dt \right) \\ \\ \leq
\displaystyle
2 c_S\sum_{j\in\mathcal J}\soT \Vert u_{i(j)}(t)\nhu \Vert \vert\mathcal W_j\Vert_{L^1(0,+\infty)}
+
\suM\left(\Vert\uo\nld^2+\Vert\vo\nld^2\right) \\ \\
+\dsp
 c_S\dsp\suM
\soT\Vert\u(t)\nhu \ioT\left(\Vert\psix(t)\nld^2 +\Vert\v(t)\nld^2
\right)dt  \\ \\
\dsp+ \suM \Vert U_i\Vert_{\infty} \ioT\left(\Vert{ \psii}_x(t)\Vert_2^2+\Vert \v(t)\Vert_2^2\right) dt \\ \\
\dsp+c_1\left( \suM \Vert {\Psi_i}_x\Vert_{\infty} \right)
 \left(\int_0^{+\infty}  |\mu(t)|^2  dt +  \ioT (\Vert u_x(t)\Vert_2^2+ \Vert v(t)\nhu^2) dt \right)
\end{array}$$
where $c_S$ are Sobolev constants and $c_1$ is a suitable constant. depending on  Sobolev's constants, on $ L_i$  and on the quantity $\gamma$ in Proposition \ref{uinf} .
\epr

\begin{proof}

We multiply the first equation  in (\ref{sysas}) by $\u$, the second one by $\v$ and we sum them; after summing up for $i\in\M$, we obtain the claim,
taking into account that from  Proposition \ref{uinf} we have
$$ \ioT \iIi u_i^2(x,t)dx dt
 \leq c \ioT\left(|\mu(t)|^2 +  \Vert u_x(t)\Vert_2^2+ \Vert v(t)\nhu^2\right) dt$$
where $c$ is a suitable constant depending on $L_i$, $\gamma$ and Sobolev constants, 
and that 
the transmission conditions (\ref{TC}) imply condition (\ref{444}) holding for $u$ and $v$, hence the sum of the terms at nodes 
is non positive .\end{proof}


\bp\label{II}
Let 
$(u,v,\psi)$ be the local solution to (\ref{sysas})-(\ref{mut}),(\ref{KC})-(\ref{compat});
 then
$$\begin{array}{ll}\displaystyle
\suM\left(\soT \Vert \vx(t)\nld^2+\soT\Vert\vt(t)\nld^2 +2 \beta_i \int_0^T\Vert\vt(t)\Vert_2^2 dt \right)\,
\\ \\ \displaystyle \leq \suM \left(\Vert\vox\nld^2+\Vert\vt(0)\nld^2\right)+\dsp + 2
c_S \suMjj  \soT\Vert u_{i(j)}(t)\nhu \Vert  |\mathcal W''_j\Vert_{L^1(0,+\infty)}
\\ \\ \dsp  +2c_S  \suMj   \left(\Vert \mathcal W_j'\Vert_{L^{\infty}(0,+\infty)} \soT \Vert u_{i(j)}(t)\nhu +| \mathcal W_j'(0)| \Vert u_{i(j)}(0)\nhu \right) 
\\ \\ \displaystyle +  \suM
\left(c_S\soT\Vert \u(t)\nhu  +\Vert U_i\Vert_\infty\right)
\ioT\left(\Vert\psi_{ixt}(t)\Vert_2^2+\Vert\vt(t)\Vert_2^2\right)\, dt
\\ \\\dsp + \left (c_S \soT\Vert \psi_{x}(t)\nhu+\suM\Vert \Psix\Vert_\infty\ \right)
\ioT\left(\Vert\vt(t)\Vert_2^2+\Vert{\v}_x(t)\Vert_2^2\right)\, dt
\end{array}$$
where $c_S$ are Sobolev constants.
\epr
\begin{proof}
We set
$\del f(x,t)=f(x,t+h)-f(x,t)$; we have, for $i\in \M$,
$$
\left\{ \begin{array}{l}
\left(\del \ut +\lambda_i\del \vx\right)\del\u=0\ ,
\\ \\ \left(\del\vt+\lambda_i \del \ux\right)\del \v= \left(\del(\u  \phx )+U_i \del \psi_{ix} + \Psi_{ix} \del u_i -\beta_i \del\v \right)\del\v\ .
\end{array}\right.  
$$
  Summing  the above two equations and  integrating over
$\Ii\times(\delta,\tau)$, for  $0<\delta<\tau< T$,  $\vert h\vert \le \min\{\delta, T-\tau\}$ we obtain
\be\begin{array}{ll}\label{1}\displaystyle
\int_\delta^{\tau} \iIi
\pa_{t}\left(\frac {(\del\u)^2 +(\del \v)^2}2\right) dx \,dt+ 
\int_\delta^{\tau} \iIi\lai\pa_x\left( \del\v\del\u\right)\ dx\, dt \\ \\ \displaystyle
= \int_\delta^{\tau} \iIi\left( (\del(\u{\psi_i}_x)+\Psi_{i_x}\del \u + U_i\del\psi_x)\del\v -\beta_i(\del\v)^2\right)\ dx \,dt\ .
\end{array}
\ee

Using  condition (\ref{444})  and the boundary conditions we can compute
$$\dsp-\suM
\int_\delta^{\tau} \iIi\lai\pa_x\left( \del\v(x,t)\del\u(x,t)\right)dxdt \leq 
-\dsp \suMjj \int_\delta^\tau \del u_{i(j)}(e_j, t) \del \mathcal W_j(t) dt 
\ $$
$$=- h\suMjj \int_0^1 (\del\mathcal W_j(\tau) u_{i(j)}(e_j, \tau +\theta h) - \del\mathcal W_j(\delta) u_{i(j)}(e_j, \delta +\theta h)) d\theta $$
$$+ h\suMjj \int_0^1 \int_\delta^\tau  \del \mathcal W_j'(t) u_{i(j)}(e_j,t+\theta h) dt d\theta$$
so that, after  dividing  
  the equalities (\ref{1}) by $h^2$, summming  them for $i\in\M$ and  letting first $h$ and then  $\delta$ go to zero, we obtain the claim.
\end{proof}


\bp\label{III}
Let 
$(u,v,\psi)$ be the local solution to (\ref{sysas})-(\ref{mut}),(\ref{KC})-(\ref{compat});
 then
$$\begin{array}{ll}\displaystyle
\suM\soT\lambda_i\Vert \ux(t)\nld^2\leq   \suM \frac 2 {\lambda_i}\left(\soT\Vert\vt(t)\nld^2+\beta_i^2\soT\Vert\v(t)\nld^2\right)\\ \\
\displaystyle
+ \suM(c_S\soT\Vert \u(t)\nhu + \Vert U\Vert_\infty) \left(\soT\Vert\ux(t)\nld^2 +\soT\Vert \phx(t)\nld^2\right)\\ \\
+\dsp  \dsp \suM\Vert \Psix\Vert_\infty  \soT\Vert\u(t)\nhu^2 
\end{array}$$
where $c_S$ depends on Sobolev constants .
\epr
\begin{proof}
We multiply the second equation in (\ref{sysas}) by $\ux$, we integrate over $\Ii$ and we sum for
$i\in\M$; using the Cauchy-Schwartz inequality  we obtain the claim.
 \end{proof}


\bp \label{IV}
Let 
(\ref{nd}) hold and 
let $(u,v,\psi)$ be the local solution to (\ref{sysas})-(\ref{mut}),(\ref{KC})-(\ref{compat});
 then
$$\begin{array}{ll}\displaystyle
\suM\lambda_i\ioT\Vert \ux(t)\nld^2\, dt\leq \suM  \frac 2 {\lambda_i}\ioT\left(\Vert\vt(t)\nld^2+\beta_i^2\Vert\v(t)\nld^2\right)
\,dt
\\ \\
\displaystyle
+ \suM(c_S \soT\Vert \u(t)\nhu+ \Vert U\Vert_\infty)
\ioT\left(\Vert\ux(t)\nld^2 +\Vert \phx(t)\nld^2\right)\,dt \\ \\
\dsp +c_2  \suM \Vert \Psix\Vert_\infty\left ( \int_0^{+\infty}  |\mu(t)|^2 dt +\ioT\left(\Vert \ux(t)\nld^2 +\Vert \v(t)\nhu^2\right) dt \right)
\end{array}$$
where  $c_S$ depends on Sobolev constants  and $c_2$ is a 
constant depending on Sobolev constants, on  $L_i$  and on $\gamma$ .
\epr
\begin{proof}
We multiply the second equation (\ref{sysas}) by $\ux$, we integrate over $\Ii\times(0,T)$ and we sum for
$i\in\M$; using the Cauchy-Schwartz inequality and Proposition \ref{uinf}, we obtain the claim.
 \end{proof}


\bp\label{V}
Let 
 $(u,v,\psi)$ be the local solution to (\ref{sysas})-(\ref{mut}),(\ref{KC})-(\ref{compat});
then
$$\begin{array}{ll}\displaystyle
\ \suM\lambda_i^2\ioT\Vert \vx(t)\nld^2 \,dt\leq 
c_3 \suM\left(\Vert \vo\nld^2+ \Vert \uo\nhu^2\left(1+\Vert \pho\nhu^2\right)
\right)\\ \\  \dsp
+ \suMjj 	c_S\left( \Vert u_{i(j)}(0)\nhu  \vert\mathcal W_j(0)\vert 
\ + \Vert\mathcal W_j\Vert_{L^\infty(0,+\infty)} \soT\Vert u_{i(j)}(t)\nhu\right)\\ \\
+
\dsp  c_S \suMjj \soT\Vert u_{i(j)}(t)\nhu \Vert \mathcal W_j'\Vert_{L^1(0,+\infty)} 
 \dsp  
+c_4 \suM\left(\soT\Vert \vt(t)\Vert_2^2+
\ioT\Vert v_{it}(t)\nld^2 dt \right)
 \\ \\
\displaystyle
\displaystyle
+\frac 1 2 \suM 
\left( c_S \soT\Vert \u(t)\nhu+\Vert U\Vert_\infty\right) \ioT\left(\Vert \v(t)\Vert_2^2 + \Vert \psi_{i_{xt}}(t)
\nld^2\right)\,dt\\ \\
\dsp 
+\frac {1} 2 c_S \suM  \left( \soT\Vert\phx(t)\nhu +
\Vert \Psi_x\Vert_\infty\right)
\ioT\Vert \v(t)\nhu^2 \,dt
\end{array}$$
where $c_S$ are Sobolev constants, and $c_3,c_4,c_5$ are positive constants depending on $\lambda_i, \beta_i, \sigma_{ij}$, 
and on Sobolev constants.
\epr
\begin{proof}
Using the same notations as in the proof of Proposition \ref{II}, by the second equation in (\ref{sysas})
we obtain, for $0<\delta<\tau<T$, $\vert h\vert \le \min\{\delta, T-\tau\}$,
$$\begin{array}{ll}\displaystyle
\int_\delta^\tau \iIi\left((\v\del \v)_t-\vt\del \v -\lai \vx\del\u\
+\lai\left(\v\del\u\right)_x \right) dx\ dt\\ \\
=
\dsp\int_\delta^\tau\iIi \v\left(\del(\u\phx)+\del u \Psix +U_i\del\phx-\beta_i \del\v\right)\ dx \ dt\ .\end{array}$$
Using the boundary conditions in (\ref{sysas})  and (\ref{TC}) 
 we can write
\be\label{2b}\begin{array}{ll}\displaystyle
\suM \frac 1 h \int_\delta^\tau \iIi 
\left(-\lai \vx\del\u
+
\beta_i \v\del\v\right)\ dx\ dt  \\ \\ \dsp
=
\frac 1 h\suM
\iIi\left( -\v(\tau)\del \v(\tau)\ dx  +\v(\delta)\del \v(\delta)\right) dx 
\dsp - \frac 1 h \int_\delta^\tau \suMjj \mathcal W_j \del u_{i(j)}(e_j)  dt 
 \\ \\ \dsp
+\frac 1 h\suM\int_\delta^\tau \iIi\left(\vt\del \v 
+ \v\left(\del(\u\phx)+\Psix\del u_i+U\del \phx )\right)\right) dx \ dt\\ \\ \dsp 
-\frac 1 h \int_\delta^\tau
\sum_{\nu\in\mathcal N}\sum_{i,j\in\M^\nu} \frac {\sigma_{ij}} 2\left(u_j(N_\nu)-\u(N_\nu)\right)
\del \left(u_j(N_\nu)-\u(N_\nu)\right)dt .
\end{array}\ee
In order to treat the  terms  at the inner nodes, as in \cite{noi},  we set 
$H(t)=u_j(N,t)-\u(N,t)$,  and we have
$$\begin{array}{ll}\dsp
\lim_{h\to 0}\frac 1 h \int_\delta^\tau H(t)\del H(t)\ dt
=\frac 1 2 \left(H^2(\tau)-H^2(\delta)\right) \ .\end{array}$$
As regard to the terms at the boundary nodes, we argue as in the proof of Proposition \ref{II}.
Then
we obtain the claim letting $h$ and then $\delta$ go to zero and $\tau$ go to $T$ in (\ref{2b}).
\end{proof}

\medskip


\bp\label{VII}
Let 
$(u,v,\psi)$ be the local solution to (\ref{sysas})-(\ref{mut}),(\ref{KC})-(\ref{compat});
 then
$$\begin{array}{ll}\dsp
\suM\soT\left(\Vert {\psi_i}_t (t)\nld^2 +\ioT\left(b_i\Vert {\psi_i}_t(t)\nld^2+2 D_i\Vert {\psi_i}_{tx}(t)\nld^2\right) \ dt\right) \\ \\ \leq 
\dsp
c_6\suM \left(\Vert \psi_{0i}\nhd^2+
\Vert \uo\nld^2\right)
+\dsp
\suM \frac {a_i^2}{b_i}\ioT \Vert \ut(t)\nld^2 \ dt
 \\ \\
+\dsp c_S\sum_{j\in\mathcal J}  \Vert \mathcal P'_j\Vert_{L^2(0,+\infty)}\left(
 \left( \int_0^T \Vert{ \psi_{i(j)}}_t(t)\Vert_2^2 dt\right)^{\frac 1 2}  + \left( \int_0^T \Vert { \psi_ {i(j)}}_{tx}(t)\Vert_2^2 dt\right)^{\frac 1 2}\right)
\end{array}$$
where  $c_6$ depends on  $D_i,b_i, a_i$ and $c_S$ depends on    Sobolev constants. 
\epr
\begin{proof}
From the third equation in (\ref{sysas}), using  (\ref{coer}), and the boundary conditions in (\ref{sysas}), we obtain, for $0<	\delta<\tau<T$
$$\suM\left( \iIi  {(\del \psii(\tau))^2} dx +2\int_\delta^\tau\iIi\left( b_i(\del \psii)^2+D_i (\del {\psii}_x)^2 \right) dx dt\right)$$ $$ \leq
\suM\left(\iIi {(\del {\psii}(\delta))^2} dx +2a_i\int_\delta^\tau\iIi \del u_i \del \psi_i dx dt \right) + 2
\sum_{j\in\mathcal J}\int_\delta^\tau \del \mathcal P_j \del\psi_{i(j)}(e_j) dt  .$$
Since
$$   \begin{array}{ll} \dsp\sum_{j\in\mathcal J}\int_\delta^\tau \del \mathcal P_j(t) \del\psi_{i(j)}(e_j,t) dt \\ \\ \leq
\dsp\sum_{j\in\mathcal J}  \sqrt{2} c_S \Vert \Delta^h \mathcal{ P}_j\Vert_{L^2(0,+\infty)} \left( \int_\delta^\tau\left( \Vert \Delta^h\psi_{i(j)}(t)\Vert_2^2+
\Vert {\Delta^h\psi_{i(j)}}_x(t)\Vert_2^2\right) dt\right)^{\frac 1 2},\end{array}$$
as in the previous proof we can conclude dividing by $h^2$, letting $h$ go to zero and then $\delta$ go to zero and $\tau$ go to $T$.
\end{proof}

\bp\label{VIII}
Let 
$(u,v,\psi)$ be the local solution to (\ref{sysas})-(\ref{mut}),(\ref{KC})-(\ref{compat}); 
 then
 $$\begin{array}{ll}\dsp
\suM\left( \frac{D_i^2}{b_i}\soT \Vert {\psi_i}_{xx}(t)\nld^2 + 2D_i \soT \Vert {\psi_i}_x(t)\nld^2\right)  \\ \\ \leq 
\displaystyle \suM \ \left( \frac  2 {b_i} \soT\Vert {\psi_i}_t(t)\nld^2 +\frac{2 a_i^2}{b_i}\soT \Vert \u(t)\nld^2\right)  \\ \\
\dsp+
2c_S\suMjj \soT\Vert \psi_{i(j)}(t)\nhu \Vert \mathcal P_j \Vert_{L^\infty(0,+\infty)} ;
\end{array}\ $$
moreover, if (\ref{nd}) holds,
 $$\begin{array}{ll}\dsp
\dsp
\suM\int_0^T\left( \frac{D_i^2 }{b_i} \Vert {\psi_i}_{xx}(t)\nld^2 + 2D_i  \Vert {\psi_i}_x(t)\nld^2\right) dt
 \leq    \suM \ioT \frac  2 {b_i} \soT\Vert {\psi_i}_t(t)\nld^2  dt\\ \\
\displaystyle +c_7\suM  \ioT\left(\mu(t)^2+\Vert \ux(t)\nld^2 +\Vert \v(t)\nhu^2\right) dt 
+c_8 \suM\dsp \Vert \mathcal P_j\Vert_{L^2(0,+\infty)} 
\Vert \mu(t) \Vert_{L^2(0,+\infty)}\\ \\
+c_9\dsp \suMj\dsp \Vert \mathcal P_j\Vert_{L^2(0,+\infty)} 
\left( \int_0^{T}\left( 
\Vert {\psi}_x(t)\Vert_{H^1}^2+
\Vert {\psi}_t(t)\Vert_{2}^2 +\Vert {u}_x(t)\Vert_{2}^2 + \Vert {v}(t)\Vert_{H^1}^2\right)  dt\right)^{\frac 1 2 }
\end{array}\ $$
where $c_7,c_8,c_9$ are positive constants depending on $\gamma, L_i, a_i,b_i,D_i$ and Sobolev constants.
\epr
\begin{proof}

The first inequality can be achievd multiplying the third equation in (\ref{sysas}) by $\frac{D_i}{b_i}\psi_{ixx}$, 
integrating on $I_i$, summing for $i\in\M$ and using the Cauchy-Schwartz inequality and (\ref{coer}).
 Integrating over $I_i\times (0,T) $
and using  Proposition  \ref{uinf} 
we obtain the second inequality .
\end{proof}

Now we introduce the functional
$$\begin{array}{ll}
\dsp F_T^2(u,v,\psi) := \sup_{t\in[0,T]}  \Vert u(t)\Vert_{H^1}^2+
\sup_{t\in[0,T]}\Vert v(t)\Vert_{H^1}^2+
\sup_{t\in[0,T]}\Vert \psi(t)\Vert_{H^2}^2 \\ \\
\dsp +\int_0^T 
\left( \Vert u(t)\Vert_{H^1}^2+
\Vert v(t)\Vert_{H^1}^2+\Vert v_t(t)\Vert_{2}^2+\Vert \psi(t)\Vert_{H^2}^2 + \Vert\psi_t(t)\Vert_2^2
+ \Vert \psi_{xt}(t)\Vert_{2}^2\right)\ dt \ .\end{array}$$

The a priori estimates in the previous propositions allows to prove the following theorem.

\bt \label{4.1}
Let  (\ref{nd}) hold. Let $(U(x),V(x),\Psi(x))$ be a stationary solution to problem 
(\ref{sysi}), (\ref{KC}), (\ref{TC}), (\ref{ub}), (\ref{um})
and let   $(u,v,\psi)$ be the  solution to problem (\ref{sysas})-(\ref{mut}), (\ref{compat}).
There exists $\ep_0>0$ such that, if 
$$\Vert U\Vert_\infty+ \Vert \Psi_x\Vert_{\infty}\leq\ep_0\ ,$$
then,  if 
the quantities
$$\Vert \mathcal P\Vert_{H^{1}(0,+\infty)}, \Vert \mathcal W\Vert_{W^{2,1}(0,+\infty)}, \Vert \mu\Vert_{L^{2}(0,+\infty)},
\Vert u_0\Vert_{H^1}, \Vert v_0\Vert_{H^1}, \Vert \psi_0\Vert_{H^2}$$
are suitably small,  
then $F_T(u,v,\phi)$ is bounded, uniformly in $T$ ,
$$u,v\in C([0,+\infty);H^1(\A))\cap C^1([0,+\infty); L^2(\A))$$
$$\psi\in C([0,+\infty);H^2(\A))\cap C^1([0,+\infty); L^2(\A))
\cap H^1((0,+\infty); H^1(\A))$$
and, for all $i\in\M$,
$$
 \lim_{t\to +\infty} \suM\Vert u_i(\cdot, t)\Vert_{C(\ov I_i)} ,
\displaystyle \lim_{t\to +\infty} \suM\Vert v_i(\cdot, t)\Vert_{C(\ov I_i)} \ ,
\displaystyle \lim_{t\to +\infty}\suM \Vert \psi_i(\cdot, t)\Vert_{C^1(\ov I_i)}= 0\ .
$$
\et

\begin{proof}

Using the estimates proved in Propositions \ref{uinf}-\ref{VIII}, 
it is easy to prove 
the following inequality
\be \label{cubic}
F_T^2(u,v,\phi) \leq \ov C_0 
+\ov C_2 F^2_T(u,v,\psi) + \ov C_3 F_T^3(u,v,\psi)  ,\ee
where
$$\begin{array}{ll} 
\ov C_0
= c_0 \Big(  \Vert u_0\nhu^2 \left( 1+\Vert \psi_0
\nhd^2 \right)+\Vert \psi_0\nhd^2 +
\Vert v_0\nhu^2
+ \left(\Vert \Psi_x\Vert_{\infty}+1\right) \Vert \mu\Vert_{L^2(0,+\infty)}^2 \\ \\+
\Vert \mathcal P\Vert_{L^2(0,+\infty)}\Vert \mu\Vert_{L^2(0,+\infty)}
+\dsp \frac 1 {2\delta} \left( \Vert \mathcal W\Vert^2_{W^{2,1}(0,+\infty)}+ 
\Vert \mathcal P\Vert^2_{H^1(0,+\infty)}\right)\Big), \\ \\
\ov C_2=c_2\left( \Vert U\Vert_\infty+\Vert \Psi_x\Vert_\infty  +\dsp \frac \delta 2\right)\ , \\ \\
\ov C_3= \ov C_3(c_S, \beta_i,\lambda_i,b_i,a_i,D_i,L_i,\gamma, \sigma_{ij})\  > 0,
\end{array}$$
 $c_i> 0$ depend on $c_S, \beta_i,\lambda_i,b_i,a_i,D_i,L_i,\gamma$ and $\delta$ is any positive quantity.

If we choose $\Vert U\Vert_\infty, \Vert \Psi_x\Vert_\infty, \delta $  in such a way that $\ov C_2<1$, and  we choose $\ov C_0$  small enough, if in addition $F_0(u,v,\psi)< \dsp \frac {2(1- \ov C_2)} {3 \ov C_3}$, then the inequality (\ref{cubic}) implies  that $F_T(u,v,\psi)$ remains  uniformly bounded for  all $T>0$; then the solution  is globally defined.
Moreover 
the set $\{u(t),v(t),\psi(t)\}_{t\in[0,+\infty)}$ is uniformly bounded in $(H^1(\A))^2\times H^2(\A)$; thus,  if we call $E_s$ the set of accumulation points of $\{u(t),v(t),\psi(t)\}_{t\geq s}$ in $(C(\A))^2\times C^1(\A)$, then $E_s$ is not empty and  $\displaystyle E:=\cap_{s\geq 0} E_s \neq \emptyset$.
Let $\hv(x)$  be such that, for a sequence $t_n\to +\infty$,
$$\displaystyle \lim_{n\to +\infty} \suM\Vert v_i(\cdot, t_n)- \hv_i(\cdot)\Vert_{C(\ov I_i)}  =0\ .$$
If we set 
$\omega_i(t):=\Vert  v_i(t,\cdot)\Vert_{L^2(I_i)}$ then the estimates obtained for the functions $v_i$ imply that 
$\omega_i\in H^1((0,+\infty))$  and, as a consequence, $\displaystyle  \lim_{t\to+\infty} \omega_i(t) =0$.
As 
$\displaystyle \lim_{n\to +\infty}\Vert v_i(\cdot,t_n)\Vert_2=\Vert \hv _i(\cdot)\Vert_2$, we obtain $\Vert \hv\Vert_2 = 0$.
The same argument can be applied to the functions $u_{i}$ and $\psi_i$.
\end{proof}

\end{section}

\begin{section} 
{\bf Stationary solutions on acyclic networks} 

In this section we study the real existence of stationary solutions to problem 
 (\ref{sysi})-(\ref{nd}).
 Concening the uniqueness, we can notice that the results of the previous section imply that two stationary solutions with the same mass and the same boundary data , which are small in 
 $H^1\times H^1\times H^2$ norm, have to coincide.

In this section 
we restrict our attenction to acyclic graphs and we approach the study of existence of stationary 
solutions $(U(x),V(x), \Psi(x))$  with mass
\be\label{mu0}\mu_s=\suM \iIi U_i(x)\ dx ,\ee
and  boundary data
\be\label{bi} \eta_j \lambda_{i(j)} V_{i(j)}(e_j)=W_j\ ,\qquad\eta_j\partial_x \Psi_{i(j)}(e_j)+d_j \Psi_{i(j)}(e_j)=P_j\ ,\quad j\in\mathcal J\   ,
\ee
assuming conditions  (\ref{nd}) and some suitable smallness conditions  on $|\mu_s|$, $|W_j|$ and $|P_j|$ .

Of course, for all $i\in\M$,  $V_i(x)$ is a constant function, $V_i(x)=V_i$; moreover, we recall that a set
 of  boundary data $\{W_j\}_{j\in\mathcal J}$ is compatible with the transmission conditions only  if 
 $\dsp\sum_{j\in\mathcal J}  W_j  =0$ (see previous section).
These  facts holds true for general networks. 

In  the case of acyclic network, a set of admissible boundary values  $\{W_j\}_{j\in\mathcal J}$ determines univokely  the costant  value of each function $V_i$  on the internal arc $\Ii$. 
Actually, let consider
an internal  arc $I_{\iota}$  and its starting  node $N_\eta$ and  the sets
\be\label{sets}\begin{array}{ll} \mathcal Q=\{ \nu\in \mathcal N: N_\nu \textrm{ is linked to } N_\eta  \textrm{ by a path not covering }  I_\iota\} \ ,
\\
 \mathcal J' =\{ j\in \mathcal J: e_j \textrm{ is linked to } N_\eta  \textrm{ by a path not covering }  I_\iota\} \ ;
\end{array}\ee
at each  inner node the conservation of the flux (\ref{consv})  holds, then
$$\displaystyle \sum_{\nu\in\mathcal Q\cup \{\eta\}} \left( \sum_{i\in I^\nu} \li V_i(N_\nu) -\sum_{i\in O^\nu} \li V_i(N_\nu)\right) =0\ .$$
Since  $V_i(x)$ is constant on  $I_i$  for all $i\in\M$, using the first condition in (\ref{bi}),
 the above equality reduces to
\be\label{Vl} \lambda_\iota V_\iota(x)=- \dsp\sum_{j\in\mathcal J'} W_j\ .\ee

Hence, a stationary solution to problem (\ref{sysi})-(\ref{TC}) satisfying (\ref{mu0}) is a triple $(U(x),V(x),\Psi(x))$ where $V$ is determined by the boundary conditions and the functions $U$ and $\Psi$ solve the following problem.

\medskip

{\it 
Find $C_i$, $i=1,...m$, and $\Psi\in H^2(\mathcal A)$ such that}
\be\label{pr-ell22}\left\{\begin{array}{ll}\displaystyle
-D_i \partial_{xx} \Psii(x)+b_i \Psii(x) = a_i  U_i(x) \qquad\qquad \qquad x\in I_i \ ,\quad i\in\M \ ,\\ \\
U_i(x)=\exp\left(\frac {\Psii(x)} \li\right) \left( C_i -\dsp\frac{\beta_i}{\li} V_i\dsp\int_0^x\exp\left(-\frac {\Psii(s)} \li\right) \ ds\right)\ ,
\\ \\
\eta_j D_{i(j)}\partial_x{\Psi_{i(j)}}(e_j)+d_j \Psi_{i(j)}(e_j) =P_j\ , \qquad  j\in\mathcal J,  \\ \\
\delta_\nu^i\dsp D_i\partial_x\Psii(N_\nu)=\suMjn \alpha^\nu_{ij} (\Psi_j(N_\nu)-\Psii(N_\nu))\ ,\quad i\in\M^\nu,\nu\in\mathcal N     
,\\ 
-\delta_\nu^i \li V_i =  \dsp  \suMjn \sigma^\nu_{ij} (U_j(N_\nu)-U_i(N_\nu))\ \ ,\quad\qquad i\in\M^\nu,\nu\in\mathcal N ,   
 \\    
\displaystyle\suM \iIi U_i(x)dx =\mu_s\ .
\end{array}
\right.
\ee

\medskip

We are going  to prove  existence of solutions to problem (\ref{pr-ell22}) using a fixed point technique; we need some preliminary results.

Given $f_i\in H^2(I_i)$, for $i\in\M$,   we introduce the functions
$$E^f_i(x) =\exp\left(\frac {f_i(x)}{\li}\right)\ ,\qquad J^f_i(x)= \dsp\frac {\beta_i}{\li}\int_0^x \exp\left(-\frac {f_i(s)}{\li}\right) ds\ .$$

\begin{lemma}\label{lemma0} Let $\mathcal G$ an acyclic graph and let (\ref{nd}) hold.
 Given  a function $f\in H^2(\A)$ and  real values $\mu_s$ and  $V_i$, $i\in\M$,  there exists a unique  
 $C^f=(C^f_1, C^f_2, ...,C^f_m) $ such that the functions 
$$\mathcal U^f_i(x)=\exp\left(\frac {f_i(x)} \li\right) \left( C^f_i -\dsp \frac{\beta_i}{\li} V_i\dsp\int_0^x\exp\left(-\frac {f_i(s)} \li\right) \ ds\right)$$
  satisfy
 \be\label{1q} -\delta_\nu^i \li V_i =    \suMjn \sigma^\nu_{ij} (\mathcal U^f_j(N_\nu)-\mathcal U^f_i(N_\nu))\ ,\qquad \nu\in\mathcal N  \ ,i\in\M^\nu \ , 
\ee 
\be\label{2q}\displaystyle\suM \iIi  \mathcal U^f_i(x)dx =\mu_s\ .\ee
   \end{lemma}
 \begin{proof}
The conditions (\ref{1q}) can be rewritten as 
 (\ref{ccnd}).  Using such relations at the node $N_1$,   
 for $ i\in \M^1$
we can express the coeffcients  $C^f_i$
as linear combination of   the values  $V_i$ and    $\C_{k_1}$ , where $k_1$ is the index in (\ref{nd}),
$$C^f_i = (E^f_i(N_1))^{-1} \left( E^f_{k_1}(N_1)
\left(C^f_{k_1}-V_{k_1} J_{k_1}^f (N_1)\right)
+\dsp\sum_{j\in\mathcal M^1} \gamma^1_{ij} V_j\right) + 
V_i  J_i^f (N_1)  \ .$$
Setting 
$$Q^f_{i\nu} =\dsp  E^f_{k_\nu}(N_\nu) E^f_i(N_\nu)^{-1}\ ,$$
$$ O^f_{i\nu}=  E_i^f (N_\nu)^{-1} \left(-V_{k_\nu}E^f_{k_\nu}(N_\nu) J_{k_\nu}^f (N_\nu)
+\dsp\sum_{j\in\mathcal M_\nu} \gamma^1_{ij} V_j\right)+ V_i J_i^f (N_\nu)
$$ 
we have
$$C^f_i=  Q^f_{i1} C^f_{k_1} + O^f_{i1}\ ,\ \  \  i\in\M^1 \ ;$$
now, if $N_\nu$ and $N_1$ are two consecutive nodes,  linked by the arc $I_l$, arguing as before we infer   that the coefficients $C^f_{k_\nu} $ and $C^f_{k_1}$ have to satisfy the following relation
$$C_l^f=Q^f_{l \nu} C^f_{k_\nu} + O^f_{l\nu} =Q^f_{l1} C^f_{k_1} + O^f_{l1}\ ,$$
which expresses $C_{k_\nu}^f$ in terms of $\C_{k_1}$; 
so, for all $i\in\M^\nu$, we have the expression
$$ \dsp C^f_i=  Q^f_{i\nu} C^f_{k_\nu} + O^f_{i\nu}=\frac{ Q^f_{l1} Q^f_{i\nu}}  {Q^f_{l\nu}}  C^f_{k_1} +
 Q^f_{i\nu}\frac { O^f_{l1}-O^f_{l\nu}} {Q^f_{l\nu}} +O^f_{i\nu} \ \ .
$$
Since there are no cycles in the network, iterating this  procedure we can write  univokely all the coeffcients $C^f_i$, $i\in\M$  in terms of $C^f_{k_1}$,
\be\label{inf1} \dsp C^f_i=  \tilde Q^f_iC^f_{k_1} + \tilde O^f_{i}\ , \ i\in \M\ ,\ee
where $\tilde Q^f_i$ and $\tilde O^f_i$ are suitable quantities depending on the function $f$ and on the values $V_i$.
In other words, system (\ref{1q}) has 
  $\infty^1$ solutions given by (\ref{inf1}), for $C_{k_1}\in\R$.
In order to determine $C^f_{k_1}$ we use condition (\ref{2q}):
$$\dsp C^f_{k_1}=\left( \suM \tilde Q^f_i \iIi E_i^f(x) dx \right)^{-1}
\left( \mu_s - \suM \iIi\left( \tilde O^f_i- V_i J^f_i(x)\right) E^f_i(x)
 dx\right) .
$$
\end{proof}

Now, given $f\in H^2(\A)$ we consider the problem

\be\left\{\begin{array} {ll} \label{ep}-D_i\partial_{xx}\Psii(x)+b_i\Psii(x)= a_i \mathcal U^f_i(x)\ ,  \qquad \t{ for all } i\in\M \ ,\\ \\
\eta_j\partial_x\Psi_{i(j)}(e_j)+d_j \Psi_{i(j)}(e_j) =P_j\ , \qquad  j\in\mathcal J,  \\ \\
\delta_\nu^i\dsp D_i\partial_x\Psii(N_\nu)=\suMjn \alpha^\nu_{ij} (\Psi_j(N_\nu)-\Psii(N_\nu))\ ,\quad \nu\in\mathcal N\   ,    
\end{array}\right.  \ee
which has a unique solution (see the proof of Proposition \ref{mdiss}).
 
 We set
$ \dsp \Theta:= \sum_{j\in\mathcal J} |P_j| + \mu_s \max\{a_i\}_{i\in\M}
 ;$
 then the following   estimates hold  .

\bl\label{1L} Let $\mathcal G$ be an acyclic graph.
Let  $\mathcal U^f_i(x)\geq 0$ and let $\Psi\in H^2(\mathcal A)$ be the solution to problem (\ref{ep}). Then there exist two positive constants $K_1,K_2$, depending on the parameters 
$b_i,D_i,$ $L_i, d_j$ ( $i\in\M$, $j\in\mathcal J$), such that 
\be\label{phinf0} \Vert \Psi\Vert_{\infty}\leq K_1 \Theta\ ,\ \ \ 
\Vert \Psi_x\Vert_{\infty}   \leq K_2 \Theta\ .\ee

\el

\begin{proof}

We multiply the first equation in (\ref{ep}) by $\Psii$, we integrate on the interval $\Ii$ and we sum for $i\in\M$ ; using (\ref{coer}) to treat
the terms evaluated at the internal nodes,  we obtain
$$ \suM\iIi \left(D_i\Psix^2 +b_i \Psi^2 \right) dx
\leq 
 \sum_{j\in\mathcal J} \left(P_j \Psi_{i(j)}(e_j)-d_j \Psi^2 _{i(j)}(e_j)\right)
+ \suM a_i\iIi \mathcal U^f_i |\Psii| dx
$$
$$\leq\left(\sum_{j\in\mathcal J}\vert P_j \vert + \mu_s \max\{a_i\}_{i\in\M}\right)
 \suM c_i^S\Vert \Psii\Vert_{H^1} 
  \ ,$$
where $c_i^S$ are Sobolev constants.
This yields the first inequality in the claim.

In order to obtain the second inequality, first we notice that, if $e_j$ is an external node and $I_{i(j)}$ is the corresponding  external arc, then  the following inequality holds
$$|D_{i(j)} \partial_x\Psi_{i(j)}(x)|\leq \int_{I_{i(j)}}  D_{i(j)} | \partial_{yy}\Psi_{i(j)} (y) |\ dy + |P_j-d_j \Psi(e_j)|
\ .$$
Then we consider an  internal arc $I_\iota $ and   its starting node $N_\eta$,  the sets $\mathcal Q$, $\mathcal J'$  as in (\ref{sets}) and 
$$\mathcal S=\{ i\in \mathcal M: I_\iota  \textrm{ is incident with } N_\nu  \textrm{ for some }\nu \in\mathcal Q\}\ ;
$$
at each node the conservation of the flux (\ref{contpsi}) holds, then
$$\displaystyle \sum_{\nu\in\mathcal Q\cup \{\eta\}} \left( \sum_{i\in I^\nu} D_i\Psi_x(N_\nu) -\sum_{i\in O^\nu} D_i\Psi_x(N_\nu)\right) =0\ .$$
Then, for all  $x\in I_\iota$, using the above equality and the boundary conditions (\ref{phibc}), we have
$$
\dsp D_\iota {\Psi_\iota}_x (x)
=  \sum_{i\in \mathcal S} \iIi D_i{ \Psi_i}_{yy}(y) \ dy + \int_{I_\iota} D_\iota{\Psi_\iota}_{yy}(y)\ dy -
 \sum_{j\in \mathcal J'}   \left( P_j-d_j \Psi(e_j)\right)  \ .$$
Then , for all $l\in \M$,
$$\dsp D_l |{\Psi_l}_x(x)| \leq  \sum_{j\in\mathcal \mathcal J} \vert P_j -d_j\Psi(e_j)\vert+ \suM
\iIi | b_i \Psi_i(y) - a_i \mathcal U^f_i(y) | dy $$
and using the first inequality in 
(\ref{phinf0}) we obtain
$$\max_{i\in\M} \Vert \Psix\Vert_\infty \leq \frac {\Theta}{\min\{D_i\}_{i\in\M}} \left (1+
K_1\left( \suM b_iL_i+ \dsp \sum_{j\in\mathcal J}d_j \right)
\right)\ .
$$
\end{proof}

The previous results give the tools to prove the following theorem of existence for stationary solutions under smallness conditions for some data; in particular we remark that the condition on $\dsp\suM |V_i|$ is a condition on $W_j$, $j\in\mathcal J$, thanks to (\ref{Vl}).

\bt \label{pmu2} 
Let $\mathcal G$ be an acyclic graph and let ( \ref{nd}) hold.
Let $\dsp \sum_{j\in\mathcal J} W_j  =0 $;
there exists $\epsilon>0$ and $\delta=\delta(\Theta)>0$, increasing with $\Theta$, such that, if  
$  \delta\dsp \suM| V_i|\leq \mu_s\ \t{ and } \  0\leq \mu_s+\dsp \suM| V_i| < \epsilon\ $,   then problem (\ref{sysi}), (\ref{bi}), (\ref{KC}), (\ref{TC}) has  a  stationary solution 
 $(U(x),V(x),\Psi(x))\in (C^\infty(\mathcal A))^2\times C^\infty(\A)$ satisfying (\ref{mu0}), which is the unique 
 one
 with  non-negative $U(x)$.
\et
\begin{proof}
If a stationary solution $(U(x),V(x),\Psi(x))$ exists  then $V_i$ are univokely determined by the boundary data  $W_j$, $j\in\mathcal J$; moreover  $U,\Psi$ satisfy (\ref{pr-ell22}) and $U_i(x)$ are univokely determined by $\Psi_i(x)$ and the values $V_i$, $\sigma_{ij}$ and $\mu_s$
(Lemma \ref{lemma0}). We remark that, if the solution
verifies  the properties in the claim, then  the estimates in Lemma \ref{1L} hold for $\Psi$.

Let  $G$ be the operator defined in $D(A_2)$ (see (\ref{da2}))  such that, if $\Psi^I\in D(A_2)$ then $\Psi= G(\Psi^I)$ is the solution of problem
(\ref{ep}) where  $f=\Psi^I$ and $ \mathcal U^{\Psi^I}$ is the function in Lemma \ref{lemma0}.
We consider $G$ on the set
 $$B_\Theta:= \{ \Psi\in D(A_2) : \Vert \Psi\Vert_{\infty}   \leq K_1 \Theta\ ,
\Vert \Psi_{x}\Vert_{\infty}   \leq K_2 \Theta\ \}\ ,$$
where $K_1,K_2$ are the constants in Lemma \ref{1L}, equipped with the distance $d$ generated by norm of $H^2(\A)$; $(B_{\Theta},d)$ is a complete metric space.

Using the expression of  $C_i^f$ given in the proof of Lemma \ref{lemma0}  and setting
$$\dsp\Lambda^{\Psi^I}_1:= \dsp\suMj \int_{I_j} \left( \tilde O^{\Psi^I}_j -V_jJ^{\Psi^I}_j(x) \right) E^{\Psi^I}_j(x)dx 
\ ,\ \ \ \dsp \Lambda^{\Psi^I}_2:=\suMj \tilde Q^{\Psi^I}_j \int_{I_j} E^{\Psi^I}_j(x) dx\ ,$$
we can write
\be \label{iui}
\mathcal U^{\Psi^I}_i(x) = E^{\Psi^I}_i(x) \,
\frac 
{ \tilde Q^{\Psi^I}_i( \mu_s -\Lambda^{\Psi^I}_1)
+
 ( \tilde O_i -V_i J_i(x))\Lambda^{\Psi^I}_2}
{\Lambda^{\Psi^I}_2 } 
.\ee

It is readily seen that there exist some positive quantities 
$q_i^{\Theta}$, increasing in  $\Theta$, and some positive quantities $q_i^{-\Theta}$, decreasing  in $\Theta$, depending also    on the parameters of the problem, 
such that, for all $f\in B_\Theta$, 
\be\label{j1}0< q_1^{-\Theta} \leq  E^f_i (x) \leq q_1^\Theta\ ,\  \ 0\leq   J^f_i (x) \leq q_2^\Theta\ ,\ \ \forall x\in\Ii\ , \ \ \forall i\in\M\ ,\ee
\be\label{j2}0< q_3^{-\Theta} \leq \tilde Q^f_i \leq q_3^\Theta\ , \ \ \dsp \vert \tilde O_i^f\vert \leq q^\Theta_4 \sum_{j\in\M} |V_j|\ ,
\ \  \ \forall i\in\M\ ,\ee
\be\label {j3}0<q_5^{-\Theta}\leq  \Lambda_2^f \leq q_5^\Theta\ ,\ \  \vert \Lambda_1^f\vert \leq q_6^\Theta \suMj |V_j|\ .\ee
Hence, fixed 
 $\mu_s\geq 0$ and $P_j$, $j\in\mathcal J$, it is possible to find a quantity $\delta=\delta(\Theta)$ such that,  if $\delta \dsp \suM |V_i|\leq \mu_s$
then $\mathcal U^{\Psi^I}_i(x) \geq0$ for all $i\in\M$.
This fact allows us to use Lemma \ref{1L} and infer that 
 $\Psi\in B_\Theta$. 

Now we are going to prove that, if 
$\mu_s+\dsp \suM|V_i|$ is small then $G$  is a contraction mapping in $B_\Theta$.
We consider $\Psi^I,\ov\Psi^I \in B_\Theta$ and the corresponding $\Psi=G(\Psi^I)$ and $\ov\Psi=G(\ov\Psi^I)$;
using the equation satisfied by $\Psi$ and $\ov \Psi$ 
and  (\ref{coer}), 
we infer that 
\begin{equation}\label{mai}\suM \Vert {\Psi_i} -\ov {\Psi_i}\Vert^2_{H^2} 
\leq 
K_3 
 \suM \Vert \U^{\Psi^I}_i - \U_i^{\ov \Psi^I}\Vert^2_{2} \ ,\end{equation}
 where $K_3$ is a suitable constant depending on $a_i,b_i,D_i$;
moreover, we have
\begin{equation}\label{ora}\begin{array}{ll} 
 \displaystyle
 \left \vert \U_i^{\Psi^I}(x)- \U_i^{\ov \Psi^I}(x)\right \vert
  \leq
 \dsp
 \left  \vert E^{\Psi^I}(x)- E^{\ov\Psi^I} (x)\right\vert 
  \ \frac{q_3^\Theta \mu_s+ (q_3^\Theta q_6^\Theta +q_4^\Theta +q_2^\Theta q_5^\Theta )\dsp \suMj|V_j|}
 {q_5^{-\Theta} }\\ \\ +
 \dsp\frac {q_1^\Theta}
  {({q_5^{-\Theta}})^2}
   \left(\mu_s+q_6^\Theta \suMj |V_i|\right)
    \left(q_5^\Theta   \left \vert \tilde Q^{\Psi^I} -\tilde Q^{\ov \Psi^I}\right\vert    +q_3^\Theta
  \left \vert \Lambda_2^{\Psi^I}-\Lambda_2^{\ov\Psi^I}\right\vert       \right)\\ \\
 +  \dsp q_1^\Theta
  \left( \frac{q_3^\theta q_5^\Theta} {({q_5^{-\Theta}})^2}.  \left  \vert \Lambda_1^{\Psi^I}-\Lambda_1^{\ov\Psi^I}\right\vert 
  + 
  \left \vert \tilde O^{\Psi^I}-\tilde O^{\ov \Psi^I}\right \vert    +  \left\vert V_i  \right \vert  \left\vert J_i^{\Psi^I}(x) - J_i ^{\ov\Psi^I} (x)	\right \vert \right)
 \ .\end{array}\end{equation}

It is easily seen that, for suitable positive quantities 
 $q_i^{\Theta}$,  depending on  $\Theta$  and on the parameters of the problem,
  increasing with $\Theta$, for any  $\Psi^I, \ov \Psi^I\in B_\Theta$ the following inequalities hold
$$\left\vert  E_i^{\Psi^I}(x) -{E_i^{\ov\Psi^I}}(x)
\right\vert \leq q^\Theta_{7} \vert \Psi_i^I(x) -\ov \Psi_i^I(x) \vert\ ,
\ \ \left \vert J_i^{\Psi^I}(x)- J_i^{\ov \Psi^I}(x) \right \vert \leq q^\Theta_{8} \Vert \Psi_i^I -\ov \Psi_i^I \Vert_\infty\,  \  ,$$

$$\left\vert  \tilde Q_i^{\Psi^I}- \tilde Q_i^{\ov \Psi^I}  \right  \vert \leq q^\Theta_{9} \suMj\Vert \Psi_j^I-\ov \Psi_j^I \Vert_\infty\,    ,
\ \  \ \left\vert  \tilde O_i^{\Psi^I}- \tilde O_i^{\ov \Psi^I}  \right  \vert \leq q^\Theta_{10}  \suMj|V_j| \suMj\Vert \Psi_j^I-\ov \Psi_j^I \Vert_\infty  ,
$$
$$\left\vert \Lambda_1^{\Psi^I} - \Lambda_1^{\ov\Psi^I} \right\vert \leq q^\Theta_{11} \left( \suMj |V_j|\right) \suM\Vert \Psi_i^I-\ov \Psi_i^I \Vert_\infty
\ ,
$$$$
 \left\vert \Lambda_2^{\Psi^I} - \Lambda_2^{\ov\Psi^I} \right\vert \leq 
q^\Theta_{12}  \suM\Vert \Psi_i^I-\ov \Psi_i^I \Vert_\infty
  ;
$$
they can be used in (\ref{ora})  so that (\ref{mai}) implies
$$\suM \Vert {\Psi_i} -\ov {\Psi}\Vert_{H^2} 
\leq  q^\Theta \left(\mu_s + \suMj|V_j|\right)
 \suM \Vert \Psi_i^I -\ov \Psi_i^I\Vert_{H^1} \ ,$$
  where $q^\Theta$ is a quantity increasing with $\Theta$, depending also on the parameters of the problem;
hence, for $\mu_s+\dsp \suMj|V_j|$ small enough, $G$ is a contraction mapping on $B_\Theta$.
Let $\Psi$ be the unique fixed point of $G$ in $B_\Theta$ and let $U=\mathcal U^{\Psi}$; then $(\Psi, U)$ is a solution to Problem $(\ref{pr-ell22})$  and it is the unique one verifying $U\geq 0 $. Regularity properties follow by the equations in (\ref{pr-ell22}).
 \end{proof}

\begin{rema} If 
  $W_j=0$ for $j\in\mathcal J$ then $V_i=0$  for $i\in \mathcal M$ and  $U_i(N_\nu)=U_j(N_\nu)$ for $i,j\in\M^\nu$, for all $\nu\in\mathcal N$;  in particular, if  $\mu_s\geq 0$ then $C_i\geq  0$ for all $i\in\M$, i.e. $U(x)\geq  0$. In this case the stationary solution of the previous  theorem is the unique stationary solution with mass $\mu_s$.
\end{rema} 
 
When the quantity 
$\dsp\suM |V_i|$ is not small enough respect to $\mu_s$ we do not have informations about the sign of $U_i(x)$; however, if the boundary data, $\mu_s$  and the parameters of the problem satisfy some relations, a stationary solution with mass $\mu_s$ exists.

First, if we set $\ov a:=\dsp \max\{a_i\}_{i\in\M}$,
as in
 Lemma \ref{1L} we prove that if
$\Psi\in H^2(\mathcal A)$ is  the solution to problem (\ref{ep}), then there exist two positive constants $\ov K_1,\ov K_2$, depending on the parameters $ b_i,D_i,$ $L_i, d_j$  ($i\in\M$, $j\in\mathcal J$), such that 
\be\label{pinf1}\ \Vert \Psi\Vert_{\infty}\leq \ov K_1( \ov a\Vert \U^f\Vert_1 +\dsp \sum_{j\in\mathcal J} |P_i|)
\ ,\ \ \ 
\Vert \Psi_x\Vert_{\infty}   \leq \ov K_2( \ov a\Vert \U^f\Vert_1 +\dsp \sum_{j\in\mathcal J} |P_i|)
\ .\ee
Moreover, let   $\ov \beta$, $\underline \lambda$ be as in Section 3 and $\gamma$ as in Lemma \ref{uinf}, and let 
$$\Omega:=|\mu_s|+2|\A|\left(\frac{2 \ov\beta}{\underline \lambda}|\A|+3\gamma\right)\suM|V_i|\ ;$$
 if 
$(U(x),V,\Psi(x))$ 
is a stationary solution, using 
Proposition \ref{uinf} and (\ref{pinf1}), we obtain 
$$\Vert U\Vert_1
\leq \dsp \Omega+\frac {4|\A| \dsp \sup_{i\in\M} \Vert \Psi_{i_x}\Vert_\infty }{\underline \lambda}\Vert U\Vert_1  \leq \Omega+
\frac{4|\A| \ov K_2}{\underline \lambda}\left( \ov a \Vert U\Vert_1+ \sum_{j\in\mathcal J}|P_j|\right) \Vert U\Vert_1\ .$$
Then, if 
\be\label{2sol}\dsp  1-\frac{4 |\A| \ov K_2}{\underline \lambda} \sum_{j\in\mathcal J}| P_j| > 0\ ,\ \
 \dsp \Omega < \frac {\underline \lambda}{16|\A| \ov K_2\ov a}\left (1-\dsp\frac{4 |\A| \ov K_2}{\underline \lambda}\dsp\sum_{j\in\mathcal J}|P_j|\right)^2\ ,\ee
setting
$$\dsp\mu^{\pm}:= \frac{\underline \lambda}{8|\A| \ov K_2 \ov a}
\left(1-\frac{4|\A| \ov K_2}{\underline \lambda}
\dsp\sum_{j\in \mathcal J}|P_j| \pm \sqrt{ \left(1-\dsp\frac{4|\A| \ov K_2}{\underline \lambda}
\dsp\sum_{j\in \mathcal J}  |P_j|\right)^2 - \frac{16 |\A| \Omega \ov K_2\ov a}{\underline \lambda}}\  \right) ,$$
we can conclude that $\mu^{\pm}>0$ and , if a stationary solution $(U,V,\Psi)$ exists, then 
$\dsp \Vert U\Vert_1 \leq \mu^-\ \ \t{ or }\ \ 
 \Vert U\Vert_1 \geq  \mu^+ $.
 
 So, under suitable smallness conditons for the data and $|\mu_s|$, we are able to prove the existence of  a stationary solution verifying
 $ \Vert U\Vert_1 \leq \mu^-$.

\bt \label{pmu3} 

Let $\mathcal G$ be an acyclic network and let (\ref{nd}) hold. Let 
$\dsp \sum_{j\in\mathcal J}  W_j  =0\  $ and  let (\ref{2sol}) hold;
there exists $\epsilon>0$ such that, if  $ |\mu_s|+\dsp\suM| V_i| < \epsilon$,   then problem (\ref{sysi}),(\ref{bi}),(\ref{KC}),(\ref{TC}) has  a  stationary solution $(U(x),V(x),\Psi(x))\in (C^\infty(\mathcal A))^2\times C^\infty(\A)$ 
satisfying (\ref{mu0}), which is the unique 
 one verifying 
 $\Vert U\Vert_1\leq \mu^-$.
\end{theo}

\begin{proof}
We proceed as in the proof of Theorem \ref{pmu2}:
  we set $\Theta_1:= \ov a \mu^- +\dsp \sum_{j\in\mathcal J} |P_i|$ and we consider the map $G$ defined there, on the set 
  $$B_{\Theta_1}:= \{ \Psi\in D(A_2) : \Vert \Psi\Vert_{\infty}   \leq \ov K_1 \Theta_1\ ,
\Vert \Psi_{x}\Vert_{\infty}   \leq \ov K_2 \Theta_1\ \}$$
equipped with the distance $d$ generated by norm of $H^2(\A)$; $(B_{\Theta_1},d)$ is a complete metric space.

Fixed 
$\Psi^I\in B_{\Theta_1}$, $\U^{\Psi^I}$ is still given by (\ref{iui}) and the relations(\ref{j1})-(\ref{j3}) hold, where the quantities  $q_i$ here depend on  $\Theta_1$.

Thanks to  (\ref{pinf1}), we can prove that  $\Psi=G(\Psi^I)\in B_{\Theta_1}$ if we show that  
$\Vert \U^{\Psi^I}\Vert_1\leq \mu^-$; this inequality can be achieved arguing as in the computations performed before the claim of this theorem, taking into account that
$$\li\U_{i_x}^{\Psi^I}=\U_i^{\Psi^I}{\Psi^I_i}_x- \beta_i V_i\ \ \t{ for all }i\in \M.$$

The last part  of the proof is equal to the one of 
Theorem \ref{pmu2},  using the quantity  $\Theta_1$  in place of $\Theta$, since,  
for $\Psi^I, \ov\Psi^I\in B_{\Theta_1}$ the following inequality holds
$$\suM \Vert {\Psi_i} -\ov {\Psi}\Vert_{H^2} 
\leq  q^{\Theta_1} \left(|\mu_s| + \suMj|V_j|\right)
 \suM \Vert \Psi_i^I -\ov \Psi_i^I\Vert_{H^1} \ ,
 $$
 where $q^{\Theta_1}$ depends on $\Theta_1, \beta_i,\lambda_i, L_i,\gamma_{ij}^\nu$, increases with $\Theta_1$
.
Let $\Psi$ be the unique fixed point of $G$ in $B_{\Theta_1}$ and let $U=\mathcal U^{\Psi}$; then $(\Psi, U)$ is the unique solution to Problem $(\ref{pr-ell22})$ such that $\Vert U\Vert_1\leq \mu^-$ and the claim is proved.
\end{proof}
\end{section}

\begin{section}
{\bf Global solutions}

Here we use the results of Sections 4 and 5 to prove the existence of global solutions to problem (\ref{sysi})-(\ref{compat}).
First we assume that  $\mathcal G$ is an acyclic graph, so that the  existence of some stationary solutions holds.

Let $\mu_s\geq 0$, let the assumptions of Theorem \ref{pmu2} hold and let $(U(x),V,\Psi(x))$ be the stationary solution; due to (\ref{phinf0}) we can control the size of the quantity $\Vert U\Vert_\infty +\Vert \Psi_x\Vert_\infty$ by means of  the size of $\ov a\mu_s+\dsp\sum_{j\in\mathcal J} |P_j|$, 
in order to satisfy the hypothesis in Theorem \ref{4.1}. So, such  theorem yields the following one.

We set $\dsp\mu(t):=\suM\iIi u_{0i}(x) dx - \sum_{j\in\mathcal J}  \int_0^t \mathcal { W}_j(s)ds$.

\bt\label{6.1}
Let $\mathcal G$ be an acyclic graph and let (\ref{nd}) hold. 
Let  the assumptions of Theorem \ref{pmu2} hold and 
let $(U(x),V,\Psi(x))$ be the stationary solution to problem 
(\ref{sysi}), (\ref{KC}),(\ref{TC}), (\ref{mu0}), (\ref{bi}) verifying $U(x)\geq 0$.
Then if the quantities
 $$\ov a\mu_s+\dsp\sum_{j\in\mathcal J} |P_j|\ , \ \ \Vert u_0-U\Vert_{H^1}, \Vert v_0-V\Vert_{H^1}, \Vert \psi_0-\Psi\Vert_{H^2}\ ,$$
$$\Vert \mathcal P(\cdot)-P\Vert_{H^{1}(0,+\infty)}, \Vert \mathcal W(\cdot)-W\Vert_{W^{1,2}(0,+\infty)}, 
\Vert  \mu(\cdot)-   \mu_s
\Vert_{L^{2}(0,+\infty)}
$$
are suitably small, then the problem (\ref{sysi})-(\ref{compat})
has a global solution $(u,v,\psi)$ such that
$$u,v\in C([0,+\infty);H^1(\A))\cap C^1([0,+\infty); L^2(\A))$$
$$\psi\in C([0,+\infty);H^2(\A))\cap C^1([0,+\infty); L^2(\A))\cap H^1((0,+\infty); H^1(\A)) \ .$$
Moreover, for $i\in\M$,
$$
 \lim_{t\to +\infty} \suM\Vert u_i(\cdot, t)-U(\cdot)\Vert_{C(\ov I_i)} =0\ ,
\displaystyle \lim_{t\to +\infty} \suM\Vert v_i(\cdot, t)-V(\cdot)\Vert_{C(\ov I_i)}  =0\ ,$$
$$\displaystyle \lim_{t\to +\infty}\suM \Vert \psi_i(\cdot, t)-\Psi(\cdot)\Vert_{C^1(\ov I_i)}= 0\ .
$$
\et

\smallskip

On the other hand, if the assumptions of Theorem \ref{pmu3} hold, then  the hypothesis in Theorem \ref{4.1} can be satisfied by controlling the size of $\ov a \mu^-+\dsp \sum_{j\in\mathcal J} |P_j|$; then, similarly to the above result,   we obtain the existence of  global solutions corresponding to data which are small perturbations of the stationary solution of Theorem \ref{pmu3}, assuming that $\ov a \mu^-+\dsp \sum_{j\in\mathcal J} |P_j|$ is suitably small.

\medskip

In the cases of general networks, we notice that  for any $\{P_j\}_{j\in \mathcal J}$, 
 it easy to prove the existence of the stationary solution $(U(x),V(x),\Psi(x))=(0,0,\Psi^0(x))$, where 
$\Psi^0 $ is the unique solution to problem (\ref{ep}) with $\U^f=0$.
If 
$\dsp 
\sum_{j\in\mathcal J} \vert P_j\vert $ is small enough ,  then $(0,0, \Psi^0(x))$ is the unique  stationary solution 
to problem (\ref{sysi}),(\ref{KC}) -(\ref{nd}),(\ref{mu0}) with $\mu_s=0$,  satisfying the boundary conditions
\be\label{bo}
V(e_j)=0\ ,\qquad\eta_j\partial_x \Psi(e_j)+d_j \psi(e_j)=P_j\ ,\quad j\in\mathcal J \ .
\ee

Moreover, if some particular relations  among the parameters of the problem  hold, then there exist stationary solutions constant on the whole network:
if we assume that  
\be\label{cs}\left\{\begin{array}{ll}\dsp \frac {a_i}{b_i}=r\ \ \t{ for all } i\in\M, \\
\t{ if }  P_j= 0  \t{ then } d_j= 0\ ,  \ \ \ \qquad  \qquad \ \ \qquad j\in\mathcal J\ ,    \\ 
\t{ if } P_j\neq 0 \t{ then } d_j\neq 0  \ \t{ and }
\dsp  \frac {P_j}{d_j}= r \frac{\mu_s}{|\A|}, \ \ j\in\mathcal J\ ,\end{array} \right.\ee
then, for any $\mu_s\in\R$  the triple $\dsp\left(\frac{\mu_s}{|\A|}, 0, r \frac {\mu_s}{|\A|}\right)$ is a stationary solution to 
(\ref{sysi}),(\ref{KC}) -(\ref{nd}),(\ref{mu0})  satisfying the boundary conditions (\ref{bo}).

Finally,  Theorem \ref{4.1} yields that the results of Theorem \ref{6.1} hold for general networks with  
$(U(x),V, \Psi(x))=(0,0,\Psi^0(x))$ and $\mu_s, W_j=0$, and with $(U(x),V, \Psi(x))=\dsp\left(\frac{\mu_s}{|\A|}, 0, r \frac {\mu_s}{|\A|}\right)$,  $W_j=0$ and conditions (\ref{cs}) .

\end{section}

\end{document}